\documentclass[a4paper, 11pt]{article}
\usepackage[margin = 0.8 in]{geometry}
\usepackage{amsmath, amssymb,amsthm,indentfirst,placeins,float,graphicx,booktabs, verbatim,xcolor,titlesec,mathrsfs,dsfont,multirow,subcaption,tikz}

\usepackage[hyperindex,breaklinks]{hyperref}
\hypersetup{linkcolor=blue, citecolor=red, colorlinks=true}

\usepackage{mathtools}                      
\mathtoolsset{showonlyrefs=true}          
\usepackage[round]{natbib}

\usepackage[justification=centering]{caption}

\newtheorem{theorem}{Theorem}
\newtheorem{lemma}[theorem]{Lemma}
\newtheorem{proposition}[theorem]{Proposition}

\newtheorem{definition}[theorem]{Definition}

\newtheorem{remark}[theorem]{Remark}
\newtheorem{assumption}[theorem]{Assumption}

\numberwithin{equation}{section}
\numberwithin{theorem}{section}

\newcommand\gam{\gamma}

\newcommand\Lc{\mathcal{L}}

\newcommand\Uc{\mathcal{U}}

\newcommand{\Eb}{\mathbb{E}}
\newcommand{\Fb}{\mathbb{F}}

\newcommand{\Pb}{\mathbb{P}}

\newcommand{\Rb}{\mathbb{R}}

\newcommand{\dd}{\mathrm{d}}

\newcommand{\oc}{\overline{c}}

\allowdisplaybreaks

\begin{document}
	
	\title{Optimal Dividend, Reinsurance, and Capital Injection Strategies \\ for an Insurer with Two Collaborating Business Lines}
	
\author{
Tim J. Boonen\thanks{Department of Statistics and Actuarial Science, School of Computing and Data Science, University of Hong Kong, China. Email: \texttt{tjboonen@hku.hk}.}
\and
Engel John C. Dela Vega\thanks{Department of Statistics and Actuarial Science, School of Computing and Data Science, University of Hong Kong, China. Email: \texttt{ejdv@hku.hk}.}
\and
Bin Zou\thanks{Department of Mathematics, University of Connecticut, USA.
E-mail: \texttt{bin.zou@uconn.edu}.}
}
	
	\date{}	
	\maketitle
	
	\begin{abstract}
		This paper considers an insurer with two collaborating business lines, and the risk exposure of each line follows a diffusion risk model. The manager of the insurer makes three decisions for each line: (i) dividend payout, (ii) (proportional) reinsurance coverage, and (iii) capital injection (from one line into the other). 
        The manager seeks an optimal dividend, reinsurance, and capital injection strategy to maximize the expected weighted sum of the total dividend payments until the first ruin. 
        We completely solve this problem and obtain the value function and optimal strategies in closed form. We show that the optimal dividend strategy is a threshold strategy, and the more important line always has a lower threshold to pay dividends. The optimal proportion of risk ceded to the reinsurer is decreasing with respect to the aggregate reserve level for each line, and capital injection is only used to prevent the ruin of a business line. 
        Finally, numerical examples are presented to illustrate the impact of model parameters on the optimal strategies.
	\end{abstract}
	
	\noindent

\section{Introduction}
\label{sec:intro}

    The optimal dividend problem is a fundamental topic in actuarial science and financial mathematics. It significantly influences how financial institutions, particularly banks and insurers, allocate their resources to meet shareholders' expectations while ensuring sufficient reserves to cover potential future liabilities. The foundations of this topic were first established in a seminal work by \cite{definetti1957}, wherein the model considers a company with one business line and seeks an optimal dividend strategy to maximize the expected discounted total dividends until the \emph{ruin time}. In this work, we extend de Finetti's model by considering a company (insurer) with two business lines and by allowing the manager of the company to purchase reinsurance for each line and inject capital from one line into the other.
    
    Let us briefly recap the essential features and results of de Finetti's model. He models the insurer's (uncontrolled) reserve by a simple random walk and assumes that the only control is the dividend payout; the optimal dividend strategy is shown to be a \emph{barrier} strategy, under which dividends are paid only when the company's reserve exceeds the barrier. Various studies confirm that barrier strategies remain optimal when the insurer's reserve process follows the classical Cram\'{e}r-Lundberg (CL) model (that is, a compound Poisson model) or the diffusion model (that is, a Brownian motion with positive drift); see, for instance, \cite{gerber1969} for the CL model and \cite{gerber2004} for the diffusion model. Stochastic control theory first found applications in the study of optimal dividend in 1990s and quickly became the main toolbox. \cite{jeanblanc1995optimization} and \cite{asmussen1997} are among the earlier contributions along this direction, with both adopting the diffusion model; for the CL model, see  \cite{azcue2005}. Since then, numerous extensions and variations have emerged; we refer the reader to \cite{albrecher2009} and \cite{avanzi2009} for comprehensive reviews and to \citet[][Chapters 2.4 \& 2.5]{schmidli2007book} for a textbook discussion on the application of stochastic control theory to optimal dividend problems.

Much of the literature on optimal dividend considers an insurer with one business line (univariate reserve process), but almost all insurers operate various business lines, such as automobile insurance, property and casualty insurance, health insurance, and life insurance. This discrepancy motivates a recent strand of literature to propose a multivariate process to model the insurer's different business lines; see \citet[][Chapter XIII.9]{asmussenbook2010} for an overview of risk theory in the dynamics of multiple insurers or business lines. Recall that a typical optimal dividend problem optimizes until the ruin time; for a univariate reserve process, the definition of ruin time is unique, but that is not the case when a multivariate reserve is considered. For the latter scenario, common definitions of ruin include: (i) \emph{first} ruin time: the first time that one of the reserve levels falls below zero; (ii) \emph{last} ruin time: the first time that all of the reserve levels, not necessarily simultaneously, fall below zero; (iii) \emph{simultaneous} ruin time: the first time that all of the reserve levels are below zero simultaneously; and (iv) \emph{sum} ruin time: the first time that the total of all of the reserve levels falls below zero. Once the definition of ruin is given, the goal is often to minimize the ruin probability, or equivalently, to maximize the survival probability. Since it is challenging to obtain closed-form solutions, most papers derive asymptotic results, approximations, bounds, or viscosity solutions; see \cite{azcue2013}, \cite{ivanovs2015} and \citet{grandits2025}. However, there are exceptions, where closed-form solutions are successfully obtained, under certain model assumptions; see, for instance, \citet{avram2008} and \citet{badescu2011}.

As seen from the above discussion, the research on optimal dividend with a multivariate reserve process often aims to minimize the ruin probability. But the classical model of \cite{definetti1957} maximizes the expected total (discounted) dividends until ruin; such an objective is more challenging to handle in the multivariate case, and consequently, related results are rather limited. An earlier work on multivariate optimal dividend problems is \cite{czarna2011}, where the reserve levels are modeled as compound Poisson processes. They consider two types of strategies: a barrier strategy with the initial reserve level serving as the barrier and a \emph{threshold} strategy, wherein dividends are paid out continuously at a fixed rate whenever the reserve level exceeds the threshold value, which is the weighted sum of the reserve levels. 
The majority of the subsequent work on multivariate optimal dividend payout problems uses compound Poisson processes to model reserve levels \citep[see][]{liucheung2014, albrecher2017, azcue2019, azcuemuler2021, strietzel2022}, while fewer papers explore diffusion models \citep[see][]{gu2018,grandits2019saj,yang2025}.

In addition to finding an optimal dividend strategy, it is of interest for insurers to manage some of the risk in the form of reinsurance. Managing the risk exposure in the multivariate case via reinsurance has also been explored. Proportional reinsurance in the form of sharing payments in fixed proportions for every incoming claim is studied in \cite{czarna2011} (for equal proportion) and in \cite{liucheung2014,azcue2019,strietzel2022} (for a general proportion). \cite{yang2025} further solve for the optimal reinsurance proportion levels in the multivariate case.

In the case of multivariate risks, capital injection from one business line to another becomes an important decision, since it can potentially save stressed business lines by using available capital from solvent business lines with adequate reserves. Such a practice is referred to as \emph{collaboration}. The studies on collaboration without transfer costs can be found in \cite{albrecher2017,gu2018,grandits2019saj}. For the case with transfer costs, see \cite{gu2018}.

With the background explained above, we now proceed to elaborate on the model and results of this paper. Our research agenda is to study an optimal dividend problem, featuring additional reinsurance and capital injection controls. To that end, we consider an insurer (insurance company) with two collaborating business lines and assume that the insurer hires one manager for both lines. The manager is in charge of making three types of decisions for each line: (1) dividend payout, (2) (proportional) reinsurance, and (3) capital injection. 
The manager's decisions are subject to some practical constraints. First, the dividend rate is bounded above by a fixed rate, called the \emph{maximum dividend rate}. Second, the proportion of risk ceded to the reinsurer is between 0 (corresponding to \emph{zero reinsurance}) and 1 (corresponding to \emph{full reinsurance}). Lastly, capital injection is modeled as a singular type of control, since such transfers are only used to save a business that would otherwise go bankrupt. 
We apply a diffusion model for the risk exposure of each business line, and we allow both positive and negative correlations between the two business lines. We assume that there is a reinsurer who is willing to offer proportional contracts for the insurer's two businesses, and it adopts the expected value premium principle to price reinsurance contracts.
The goal of the manager is to seek an optimal dividend payout, reinsurance, and capital injection strategy that maximizes the expected weighted sum of the dividend payments until the first ruin time. 
To the best of our knowledge, this is the first study on an optimal dividend problem involving bounded dividend rates, proportional reinsurance, and capital injection between two collaborating business lines under the diffusion model. 

We summarize the main contributions of this paper as follows:
\begin{enumerate}

	\item We identify three scenarios and derive, in closed form, the value function and the optimal strategy for each scenario (see Theorems \ref{prop:w0<u1<u2}, \ref{prop:u1<w0<u2}, and \ref{prop:u1<u2<w0}). The conditions for each scenario relate to whether the sum of the maximum dividend rates of the lines is ``large enough" or the maximum dividend rate of the more important line (i.e., the line with a bigger weight in the objective) is ``large enough.'' This extends the results in \cite{hojgaard1999}, which present two scenarios depending on whether the maximum dividend rate is sufficiently large.\footnote{In \cite{hojgaard1999}, the authors discuss the univariate optimal dividend payout problem with proportional reinsurance (but without capital injection) under the diffusion model.}

    	\item We show that the optimal dividend strategy, under the constraint of bounded dividend rates, is a threshold strategy, which extends the result within the univariate framework (see, for example, \cite{asmussen1997}). Compared with \cite{czarna2011}, who also study a multivariate optimal dividend problem, the threshold levels therein depend on the initial reserve levels of individual lines, but the thresholds in our paper are independent of the initial reserve levels. Instead, they depend on the weights associated with each business line in the objective; specifically, the business line with a greater weight always has a lower threshold to distribute dividends at the maximum rate.

    \item We prove that the optimal reinsurance strategy (proportion of risk ceded to the reinsurer) is decreasing with respect to the aggregate reserve level, and the two proportions are both constants simultaneously, except for the case when one business line takes a full reinsurance contract at all times.

\item On the technical side, the value function can have up to three switching points (see \eqref{eqn:g first} and \eqref{eqn:candidate 3}), and in turn, we encounter three systems of equations, together ensuring the ``smoothness'' of the value function. This significantly increases the complexity in analysis and requires a different approach from the standard univariate optimal dividend problems; see the last paragraph in Remark \ref{remark: theorem 1} for detail.

\end{enumerate}

The rest of the paper is organized as follows: Section \ref{sec:model} introduces the model and formulates the main problem. Section \ref{sec:sol} states the main results. Section \ref{sec numerical} presents the numerical examples. The proofs of the main results are included in Section \ref{section:proof of main}. Section \ref{sec:conclusion} concludes.

	\section{Model}
	\label{sec:model}
	
	We fix a complete filtered probability space $(\Omega,\mathcal{F},\mathbb{F},\mathbb{P})$, where $\mathbb{F}:=\{\mathcal{F}(t)\}_{t\geq 0}$ is a right-continuous, $\mathbb{P}$-completed filtration generated by two independent Brownian motions $B_1 = \{B_1(t)\}_{t \ge 0}$ and $B_2 = \{B_2(t)\}_{t \ge 0}$. Define two correlated Brownian motions $W_1 = \{W_1(t)\}_{t \ge 0}$ and $W_2 = \{W_2(t)\}_{t \ge 0}$ by 
	\begin{align*}
		W_1(t) = B_1(t) \quad \text{and} \quad W_2(t) = \rho \, B_1(t)+\sqrt{1-\rho} \, B_2(t),
	\end{align*}
    where $\rho \in (-1, 1)$ captures the correlation between the two Brownian motions $W_1$ and $W_2$. By the above definition, the filtration $\mathbb{F}$ is the same as the natural filtration generated by $W_1$ and $W_2$. 
		
	We study an insurer with two collaborating business lines and apply the so-called \emph{diffusion model} for the risk process of each line (see, for instance, \cite{jeanblanc1995optimization} and \cite{asmussen1997} for this model). 
	To be precise, the risk exposure of Line $i$, $i = 1, 2$, is given by 
	\begin{align}
		\label{eq:dR}
		\dd R_i(t) = \tilde{\mu}_i \, \dd t + \sigma_i \, \dd W_i(t),
	\end{align} 
    where $\tilde{\mu}_i, \sigma_i > 0$.
    Although the diffusion model in \eqref{eq:dR} may lead to a negative value of risk, it is a good approximation to the standard Cram\'er-Lundberg model (compound Poisson) when the intensity of the Poisson process (expected number of claims) is sufficiently large, which is the case for the insurance lines of most insurers \citep[see][p.98]{taksar2003optimal}. Additionally, one may assume $R_i(0) \gg 0$ and $\tilde{\mu}_i \gg \sigma_i$ so that $\Pb(R_i(t) < 0)$ is approximately zero for all $t$ \citep[see][]{liang2024two}. The insurer charges the premium for taking the risk $R_i$ by the expected-value principle, with a loading factor $\kappa_i \ge 0$; as such, the premium rate for Line $i$ is given by 
    \begin{align}
    	\label{eq:c}
    	P_i = (1 + \kappa_i) \tilde{\mu}_i, \quad i = 1, 2.
    \end{align} 
    
    We assume that the insurer hires one manager for its two business lines, and the manager makes three types of decisions regarding the operation of each line: reinsurance, capital injection between the two lines, and dividend payout. We describe each decision in detail as follows.
    \begin{enumerate}
    	\item Reinsurance decision.  The manager purchases \emph{proportional} reinsurance policies to mitigate the risk exposure of each business line \citep[see][]
        {schmidli2001optimal, taksar2003optimal}. Let $\theta_i(t) \in [0,1]$ denote the \emph{ceded proportion} of Line $i$'s risk exposure to the reinsurer and write $\theta_i = \{\theta_i(t) \}_{t \ge 0}$, $i = 1, 2$. We assume that the reinsurer also applies the same expected-value principle to determine the premium rate $\pi_i$ by 
    	\begin{align*}
    		\pi_i(t) &= (1 + \kappa_i) \, \theta_i(t) \, \tilde{\mu}_i, \quad t \ge 0.
    	\end{align*}
        Indeed, the loading for reinsurance is the same as for insurance in \eqref{eq:c}, and this assumption is referred to as ``cheap reinsurance'' \citep[see, for instance, Section 3 of ][]{luo2016optimal}.
    
    	\item Capital injection decision. Since the insurer under consideration operates two separate business lines (say automobile insurance and home insurance), we assume that the manager can inject capital from one business line to the other, \emph{without} incurring additional cost. The capital injection allows the manager to use the available resource within the company to save a business line that may otherwise go bankrupt \citep{gu2018}. Let $L_i = \{L_i(t)\}_{t \ge 0}$ denote the \emph{cumulative} amount of capital transferred into Line $i$ from Line $3-i$, and we treat $L_1$ and $L_2$ as singular-type controls.
    	
    	\item Dividend payout decision. The manager chooses a dividend strategy to distribute profits to the shareholders for each line. Let $C_i(t) \in [0, \oc_i]$ denote the dividend \emph{rate} paid at time $t$ to the shareholders of Line $i$, where $\oc_i > 0$ is the maximum rate, for $i=1,2$; we write $C_i = \{C_i(t)\}_{t \ge 0}$ as the dividend strategy for Line $i$.
    	This type of dividend strategy is called restricted (or bounded) dividend payment (see Section 2 in \cite{asmussen1997} or Case A in \cite{jeanblanc1995optimization}).  
    \end{enumerate}

\begin{remark}
	It is of interest to mention alternative choices for dividend payouts and reinsurance decisions. For instance, \citet{jeanblanc1995optimization} consider two additional types of dividend payout strategies in the univariate case: one defined by a sequence of pairs of random variables representing the timing and amount of payouts, and another involving unbounded dividend rates. Regular deterministic dividend payments, as discussed in \citet{keppo2021}, offer another interesting approach. The concept of ratcheting dividends, which imposes that dividend payments should not decrease at any time, is also worth considering, as highlighted in \citet{albrecher2022} and \citet{wang2024}.
	
	In terms of reinsurance, \citet{asmussen2000} study a problem similar to that of \citet{hojgaard1999}, but focus on excess-of loss reinsurance policies instead of proportional reinsurance. The implications of dividend payouts alongside more general reinsurance policies, such as in \citet{guan2022}, add another layer of complexity to the analysis.   
\end{remark}

For convenience, denote $u:=(\theta_1, \theta_2, L_1, L_2, C_1, C_2)$ the manager's 6-tuple reinsurance-capital injection-dividend strategy. 
For a given strategy $u$, the surplus process of Line $i$, denoted by $X_i:=X_i^u$, follows the dynamics
	\begin{align}
			\dd X_i(t) = \big[(1-\theta_i(t)) \mu_i - C_i(t)\big] \dd t - \big( 1 - \theta_i(t) \big)\sigma_i \dd W_i(t) + \dd L_i(t) - \dd L_{3-i}(t), \label{eq:SurplusProcess}		
	\end{align}
	where 
    \begin{align}
        \label{eqn:mu}
        \mu_i: = \kappa_i \tilde{\mu}_i, \quad i = 1, 2,
    \end{align}
    is the adjusted mean of the risk exposure, 
    and 
	$X_i(0) \ge 0$ is the initial surplus level of Line $i$.  
	We define the (individual) ruin time of Line $i$, $\tau_i := \tau_i^u$, as the first time that its surplus falls below zero; that is, $\tau_i$ is defined by 
	\begin{align}
		\tau_i := \inf\{t > 0 : X_i(t) < 0 \}, \quad i = 1, 2,
	\end{align}
    where $X_i$ is given by \eqref{eq:SurplusProcess}. 
    We define the first ruin time, $\tau := \tau^u$, of the insurer by 
	\begin{align}
    \label{eqn:ruin}
		\tau := \tau_1 \wedge \tau_2 = \min \{ \tau_1, \tau_2 \}.
	\end{align} 

\begin{remark}
    The first ruin time defined in \eqref{eqn:ruin}  is also introduced in \cite{czarna2011,liucheung2014,azcue2019,azcuemuler2021}. The definition of $\tau$ highlights the purpose of capital injection between lines, since we want to save the ``at-risk" line right away if the other line still has the capacity to transfer reserves without endangering itself. 
    We note that in the study of multivariate optimal dividend payout problems, there exist alternative choices of ruin, such as (i) the simultaneous ruin time, defined as $\tau_{sim}=\inf\{t > 0 : X_1(t),X_2(t) < 0 \}$ \citep{gu2018,grandits2019saj,strietzel2022} and (ii) the sum ruin time, given by $\tau_{sum}=\inf\{t > 0 : X_1(t)+X_2(t) < 0 \}$ \citep[see][]{albrecher2017}. 
    
    It is important to emphasize that the last ruin time, given by $\tau_{last}=\tau_1\vee\tau_2=\max\{\tau_1,\tau_2\}$, is not necessarily the same as $\tau_{sim}$; in fact, $\tau_{sim}\geq\tau_{last}$. To see this, suppose that Line 1 is the first line to go to ruin. Line 1 may continue to operate and its reserve level may become nonnegative again at some later time $t_0>\tau_1$ (that is, $X_1(t)>0$ for $t\geq t_0$). Line 2 may go to ruin at a later time $\tau_2$, while $X_{1}(\tau_{2})>0$. Using the definition of the last ruin time, we have $\tau_{last} = \tau_{2}$, but using the simultaneous ruin definition, it does not occur at $\tau_{2}$. 
\end{remark}

We formally define admissible strategies below.

\begin{definition}
	\label{def:ad}
	A strategy $u$ is said to be \emph{admissible} if $u$ is adapted to the filtration $\Fb$ and satisfies the following conditions:
		\begin{itemize}
			\item[(i)] $\theta_i(t) \in [0,1]$ and $C_i(t) \in [0, \oc_i]$		for $i=1,2$ and $t\geq 0$;
			
		\item[(ii)]  	$L_i$ is nonnegative, nondecreasing, and right continuous with left limits, for $i=1,2$.

	\end{itemize}
	Denote by $\Uc$ the set of all admissible control strategies.
\end{definition}

	The goal of the manager is to seek an optimal strategy that maximizes the weighted average of the dividend payouts from both lines up to the first ruin time $\tau$ defined in \eqref{eqn:ruin}. As such, we face the following maximization problem:
	\begin{equation}
		\label{eq:V}
		V(x_1,x_2):=\sup_{u \in \Uc}J(x_1,x_2;u) = \sup_{u \in \Uc}\mathbb{E}\left[a\int_0^{\tau}e^{-\beta t}C_1(t) \, \dd t+(1-a)\int_0^{\tau}e^{-\beta t}C_2(t) \, \dd t\right],
	\end{equation} 
    where $a \in [0,1]$ is a weighting factor that reflects the relative importance of Line 1 in the business operation, $\beta > 0$ is a discounting factor, and the expectation $\Eb$ is taken under $X_1(0) = x_1$ and $X_2(0) = x_2$.
    We call $J$ in \eqref{eq:V} the \emph{objective} function and $V$ in \eqref{eq:V} the \emph{value} function. 
    
    \begin{remark}
    	\cite{gu2018} also consider an insurer with two collaborating business lines, allowing frictionless capital injections from one line to the other, and their objective (see p.3 therein) is similar to ours in \eqref{eq:V}. However, they do \emph{not} allow the insurer to seek reinsurance coverage for its risk exposure, and they optimize up to the simultaneous ruin time. In addition, we adopt the classical control framework for dividend payment, while they follow the singular control framework and allow unbounded dividend rates (lump-sum payments).
    \end{remark}

    By definition, the value function $V$ is increasing in both arguments $x_1$ and $x_2$, and it satisfies the boundary condition $V(0,0) = 0$, since the first ruin occurs immediately when both lines have a zero initial surplus level. Moreover, since the discounted value of paying the maximum dividend rates $\overline c_1$ and $\overline c_2$ up to infinity is $\frac{a\overline c_1+(1-a)\overline c_2}{\beta}$, then the value function $V$ satisfies $\lim_{x_1,x_2\to\infty}V(x_1,x_2)=\frac{a\overline c_1+(1-a)\overline c_2}{\beta}$. In the rest of this section, we solve the problem in \eqref{eq:V} by the dynamic programming approach and provide a characterization of the value function $V$ via a system of HJB (HJB-variational, to be precise) equations in Proposition \ref{prop:hjb}. For that purpose, denote $\vartheta:=(\theta_1,\theta_2,c_1,c_2)$ and define the generator $\Lc^{\vartheta}(\phi)$ for some $\mathrm{C}^{2,2}$ function $\phi$ by 
	\begin{equation}\label{eq:generator}
		\begin{aligned}
			\Lc^{\vartheta}(\phi)&=\sum_{i=1}^2\left[\left[(1-\theta_i)\mu_i-c_i\right]\frac{\partial V}{\partial x_i}+\frac{1}{2}\sigma_i^2(1-\theta_i)^2\frac{\partial^2 V}{\partial x_i^2}\right]+\rho\sigma_1\sigma_2(1-\theta_1)(1-\theta_2)\frac{\partial^2 V}{\partial x_1\partial x_2}\\
			&\quad-\beta V+ac_1+(1-a)c_2.
		\end{aligned}
	\end{equation}
	
		\begin{proposition}
			\label{prop:hjb}
		The associated HJB equation for the problem in \eqref{eq:V} is given by
		\begin{equation}\label{eqn:hjborig}
			\sup\left\{\sup_{\theta_i\in[0,1], \, c_i \in[0,\oc_i]}\Lc^{\vartheta}(V), \quad \frac{\partial V}{\partial x_1}-\frac{\partial V}{\partial x_2}, \quad \frac{\partial V}{\partial x_2}-\frac{\partial V}{\partial x_1}\right\} = 0,
		\end{equation}
		with boundary condition $V(0,0)=0$.
	\end{proposition}

    \begin{proof}
	See Appendix \ref{app:hjb}.
\end{proof}

By a standard verification lemma (see, for instance, \cite[Chapter 2, Theorem 2.51]{schmidli2001optimal}), if we can find a \emph{classical} solution to the HJB equation in \eqref{eqn:hjborig} satisfying the boundary condition, then this solution is the value function $V$ to the main problem in \eqref{eq:V}, and solving the optimization problems in \eqref{eqn:hjborig} helps identify the optimal strategies. However, such a task is highly technical and involved, and we complete it in the next section.
	
\section{Analytical Solutions}
\label{sec:sol}

In this section, we study the insurer's problem in \eqref{eq:V} and obtain the optimal strategy $u^*$ and the value function $V$ by solving the HJB equation in \eqref{eqn:hjborig}. Recall from \eqref{eq:V} that $a \in [0,1]$ is the relative weight of Line 1 in the optimization, and due to symmetry between two lines, we assume, without loss of generality, that $a\leq \frac{1}{2}$ in the rest of the paper. 

We begin with a heuristic analysis of the HJB equation in \eqref{eqn:hjborig}, assuming that a classical solution $V$ exists for the moment. Since $\frac{\partial V}{\partial x_1}-\frac{\partial V}{\partial x_2} \le 0$ and  $\frac{\partial V}{\partial x_2}-\frac{\partial V}{\partial x_1} \le 0$ hold simultaneously by \eqref{eqn:hjborig}, it follows that $\frac{\partial V}{\partial x_1} = \frac{\partial V}{\partial x_2}$. As such, there exists a univariate function, $g: x \in \Rb_+ \mapsto  \Rb$, such that 
\begin{align}
    g(x) = V(x_1, x_2), \quad \text{with } x := x_1 + x_2 \ge 0.
\end{align}
In the subsequent analysis, for a bivariate function $\phi(x_1, x_2)$, we often write 
\begin{align}
    \phi(x) := \phi(x_1, x_2), \quad \text{with } x = x_1 + x_2.
\end{align}
Using the above equality on $g$, we have 
\begin{align*}
    g'(x) = \frac{\partial V}{\partial x_i}(x_1, x_2) \quad \text{and} \quad 
    g''(x) = \frac{\partial^2 V}{\partial x_i \partial x_j}(x_1, x_2), \quad i,j = 1, 2.
\end{align*}

First, we isolate the optimization over $c_i$ (dividend decision) in $\sup \, \Lc^{\vartheta}(V)$ and solve
\begin{align}
    \sup_{c_1 \in [0, \oc_1], c_2 \in [0, \oc_2]} \; \left( a - g'(x) \right) \, c_1 + \left(1 - a - g'(x) \right) c_2,
\end{align}
from which we obtain the candidate maximizer as 
\begin{align*}
 \widehat{c}_1(x)  = \begin{cases}
     0 & \text{if } g'(x) > a, \\
     \oc_1 & \text{if } g'(x) < a,
 \end{cases} 
 \quad \text{and} \quad 
 \widehat{c}_2(x)  = \begin{cases}
     0 & \text{if } g'(x) > 1 - a, \\
     \oc_2 & \text{if } g'(x) < 1 - a.
 \end{cases}
\end{align*}
Define two constants $u_1$ and $u_2$ by 
\begin{align}\label{eqn:defn of u1 & u2}
    u_1:=\inf\{u:g'(u)=1-a\} \quad \text{and} \quad 
    u_2:=\inf\{u:g'(u)=a\}.
\end{align}
We hypothesize that $g$ is a concave function ($g'' < 0$), which, along with $a \le 1/2$, implies that $u_1 \le u_2$. As such, we obtain the candidate for the optimal dividend strategies by
\begin{align}\label{eqn:optimal dividend}
    \big( \widehat{C}_1(x), \widehat{C}_2(x) \big)=
			\begin{cases}
				(0,0)&\mbox{if $x<u_1$,}\\
				(0,\overline c_2)&\mbox{if $u_1<x<u_2$,}\\
				(\overline c_1,\overline c_2)&\mbox{if $x>u_2$,}
			\end{cases},
\end{align}
where $x = x_1 + x_2$ is the aggregate surplus of the two business lines.

By a similar argument, the optimization problem regarding the reinsurance decision given an initial aggregate surplus level $x$, denoted by $\mathcal{H}(x)$, is given by 
\begin{align}\label{eqn:optim reinsurance}
    \mathcal{H}(x)=\sup_{\theta_1,\theta_2\in[0,1]} \; \sum_{i=1}^2\left[\left(1-\theta_i\right)\mu_ig'(x)+\frac{1}{2}\sigma_i^2(1-\theta_i)^2g''(x)\right]+\rho\sigma_1\sigma_2(1-\theta_1)(1-\theta_2)g''(x),
\end{align}
and we obtain the candidate for the optimal reinsurance strategies by (ignoring the constraints over $[0,1]$)
\begin{align}
\label{eqn:optimal_theta}
    \widehat{\theta}_1(x) = 1 + \frac{(\mu_1\sigma_2-\rho\mu_2\sigma_1)}{(1-\rho^2)\sigma_1^2\sigma_2} \, \frac{g'(x)}{g''(x)} 
    \quad \text{and} \quad 
    \widehat{\theta}_2(x) = 1 + \frac{(\mu_2\sigma_1 - \rho \mu_1 \sigma_2)}{(1-\rho^2)\sigma_1 \sigma_2^2} \, \frac{g'(x)}{g''(x)} .
\end{align}

Last, recall that the capital injection decision $(L_1, L_2)$ is a singular-type control; consequently, we cannot apply the first-order condition to characterize its optimal strategy as we have done for the dividend and reinsurance controls. Instead, this is achieved by analyzing the boundaries under different scenarios later.

From \eqref{eqn:optimal_theta}, we easily see that $\rho$, the correlation coefficient between the risk processes of two business lines, plays a key role in determining whether $\widehat{\theta}_1$ and  $\widehat{\theta}_2$ in \eqref{eqn:optimal_theta} can be achieved in the interior of $[0,1]$. This inspires us to discuss different cases for $\rho$ and derive the optimal strategy correspondingly. Note that for $\widehat{\theta}_i$ in \eqref{eqn:optimal_theta}, we have $\widehat{\theta}_1 \ge 1$ if and only if $\rho \ge \frac{\mu_1 \slash \mu_2}{\sigma_1\slash \sigma_2}$ and $\widehat{\theta}_2 \ge 1$ if and only if $\frac{1}{\rho} \le \frac{\mu_1 \slash \mu_2}{\sigma_1\slash \sigma_2}$, which, along with the constraints $\theta_i \in [0,1]$, implies $\widehat{\theta}_i = 1$, corresponding to \emph{full reinsurance}.

\subsection{The case of $0 < \rho< \frac{\mu_1 \slash \mu_2}{\sigma_1\slash \sigma_2}<\frac{1}{\rho}$}\label{main subsection}

In this section, we assume that $\rho$ satisfies the following conditions: 
\begin{align}
    \label{eq:rho_case1}
    0 < \rho< \frac{\mu_1 \slash \mu_2}{\sigma_1\slash \sigma_2}<\frac{1}{\rho},
\end{align}
where $\mu_i$ is defined by \eqref{eqn:mu}. For later convenience, we introduce several notations that will be frequently used in the analysis as follows:
  \begin{equation}\label{eqn:N1N2N3N4}
        \begin{aligned}
            N_1&:=(\mu_1\sigma_2-\mu_2\sigma_1)^2+2(1-\rho)\mu_1\mu_2\sigma_1\sigma_2>0,\qquad &N_2&:=N_1+2\beta(1-\rho^2)\sigma_1^2\sigma_2^2>0,\\
            N_3&:=\frac{N_1}{\sigma_2(\mu_1\sigma_2-\rho\mu_2\sigma_1)}>0,\qquad
            &N_4&:=\frac{(1-\rho^2)\sigma_1^2\sigma_2}{\mu_1\sigma_2-\rho\mu_2\sigma_1}N_3>0,
        \end{aligned}
    \end{equation}
    and 
\begin{equation}\label{eqn:w1 and w2}
                w_1:=\frac{(1-\gamma_1)(1-\rho^2)\sigma_1^2\sigma_2}{\mu_1\sigma_2-\rho\mu_2\sigma_1}>0\quad\mbox{and}\quad
                w_2:=\frac{(1-\gamma_1)(1-\rho^2)\sigma_1\sigma_2^2}{\mu_2\sigma_1-\rho\mu_1\sigma_2}>0,
	\end{equation}
    where
    \begin{equation}\label{eqn:gamma1}
        \gamma_1=1-\frac{N_1}{N_2}.
    \end{equation}

For $\widehat{\theta}_1$ in \eqref{eqn:optimal_theta}, we define $w_0$ as the zero of $\widehat{\theta}_1$; that is, $w_0$ satisfies
    \begin{equation}\label{eqn:defn of w0}
		\widehat{\theta}_1(w_0)=1+\frac{1-\gamma_1}{w_1}\cdot\frac{g'(w_0)}{g''(w_0)}=0.
	\end{equation} 
We do not discuss the existence of $w_0$ here; indeed, we will solve \eqref{eqn:defn of w0} to obtain $w_0$ in closed form below. By definition, $w_0$ can be interpreted as the aggregate surplus level at which the manager chooses \emph{zero reinsurance} for Line 1. Consequently, it serves as an important threshold for Line 1 to fully retain its risk. Because of the symmetry between the two lines, we assume, without loss of generality, that the threshold for the insurer to retain all risk associated with Line 1 is no greater than that for Line 2. This assumption is equivalent to the inequality $w_1\leq w_2$.

We state the standing assumptions for this section below.
\begin{assumption}\label{assume:w1<w2}
 Suppose that the correlation coefficient $\rho$ satisfies the inequalities in \eqref{eq:rho_case1}, and for $w_1$ and $w_2$ defined in \eqref{eqn:w1 and w2}, $w_1 \le w_2$.
\end{assumption}

To facilitate the presentation of results, we introduce the following notations:
    \begin{equation}\label{eqn:gamma234}
        \begin{aligned}
           \gamma_{2\pm}&:=\frac{-N_3\pm\sqrt{N_3^2+2\beta N_4}}{N_4},\\
            \gamma_{3\pm}&:=\frac{-(N_3-\overline c_2)\pm\sqrt{(N_3-\overline c_2)^2+2\beta N_4}}{N_4},\\
            \gamma_{4-}&:=\frac{-(N_3-\overline c_1-\overline c_2)-\sqrt{(N_3-\overline c_1-\overline c_2)^2+2\beta N_4}}{N_4},
        \end{aligned}
    \end{equation}
where $N_i$'s are defined in \eqref{eqn:N1N2N3N4}, $\beta$ is the discount rate in \eqref{eq:V}, and $\oc_i$'s are the maximum dividend rates.
In addition, we define two functions, $\psi, \zeta:(-\infty,0)\mapsto\mathbb{R}$, by 
    \begin{equation}\label{eqn:psi}
    \begin{aligned}
        \psi(z)&:=(1-a-\gamma_{3-}z)e^{\gamma_{3+} \, \zeta(z)}+\gamma_{3-}ze^{\gamma_{3-} \, \zeta(z)}-a,\\
        \zeta(z)&:=\frac{1}{\gamma_{3+}-\gamma_{3-}}\ln\left(\frac{\gamma_{3-}(\gamma_{4-}-\gamma_{3-})z}{(1-a-\gamma_{3-}z)(\gamma_{3+}-\gamma_{4-})}\right).
    \end{aligned}    
    \end{equation}

We are now ready to present the main results of this section. The zero point of $\widehat{\theta}_1$, $w_0$, in \eqref{eqn:defn of w0} plays a key role in the proofs, and its relative relation to $u_1$ and $u_2$ in \eqref{eqn:defn of u1 & u2} leads to three exclusive cases: (1) $w_0\leq u_1\leq u_2$, (2) $u_1<w_0\leq u_2$, and (3) $u_1\leq u_2<w_0$ (recall that $u_1 \le u_2$ under the assumptions of $g''<0$ and $a \le 1/2$). We obtain the optimal reinsurance and dividend strategies $(\theta_1^*, \theta_2^*, C_1^*, C_2^*)$ and the value function for each of the three cases below; see Theorems \ref{prop:w0<u1<u2}, \ref{prop:u1<w0<u2}, and \ref{prop:u1<u2<w0}. However, the optimal capital injection strategy $(L_1^*, L_2^*)$ can be obtained in a uniform way as shown in Theorem \ref{thm:op_L}. All proofs are deferred to Section \ref{section:proof of main}.
Recall that $a$ is the weight of Line 1, and $\beta$ is the discounting factor in the joint objective, $N_i$'s are defined in \eqref{eqn:N1N2N3N4}, $\gam_1$ is defined in \eqref{eqn:gamma1}, $\gam_{2\pm}$, $\gam_{3\pm}$, $\gam_{4-}$ are defined in \eqref{eqn:gamma234}, and $\oc_i$ is the maximum dividend rate of Line $i$.

	\begin{theorem}\label{prop:w0<u1<u2}
		Let Assumption \ref{assume:w1<w2} hold. Suppose (i) $\overline c_1+\overline c_2\geq \frac{\beta w_1}{\gamma_1(1-\gamma_1)}=\frac{N_3N_2}{2N_1}$ and (ii) $\psi(\alpha_{0})\leq 0$, where $\psi$ is defined by \eqref{eqn:psi} and 
        \begin{equation}
        \label{eqn:alpha_0}
            \alpha_{0}:=\frac{(1-a)\gamma_{3+}}{\gamma_{3+}-\gamma_{3-}}\left(\frac{N_3}{2\beta}-\frac{1}{\gamma_{3+}}-\frac{\overline{c}_2}{\beta}\right).
        \end{equation}
        We have the following results:
        \begin{enumerate}
            \item The zero point of $\widehat{\theta}_1$ in \eqref{eqn:defn of w0} equals $w_1$ in \eqref{eqn:w1 and w2} (that is, $w_0 = w_1$), and $u_1$ and $u_2$ defined in \eqref{eqn:defn of u1 & u2} are explicitly given by 
            \begin{equation}
                \begin{aligned}
                    u_1&=w_0+\frac{1}{\gamma_{2+}-\gamma_{2-}}\ln\left(\frac{\alpha_{2-}(\gamma_{2-}\alpha_3-1)}{\alpha_{2+}(1-\gamma_{2+}\alpha_3)}\right),\\
                    \text{and} \quad u_2&=u_1+\frac{1}{\gamma_{3+}-\gamma_{3-}}\ln\left(\frac{\alpha_{3-}\gamma_{3-}(\gamma_{4-}-\gamma_{3-})}{\alpha_{3+}\gamma_{3+}(\gamma_{3+}-\gamma_{4-})}\right),
                \end{aligned}
            \end{equation}
            where
	\begin{equation}\label{eqn:alpha2s}
		\alpha_{2+}:= \frac{w_0^{\gamma_1-1}(\gamma_1-\gamma_{2-}w_0)}{\gamma_{2+}-\gamma_{2-}}, 
        \quad 
        \alpha_{2-} := \frac{w_0^{\gamma_1-1}(\gamma_{2+}w_0-\gamma_1)}{\gamma_{2+}-\gamma_{2-}},
	\end{equation}
            \begin{equation}
            \label{eqn:alpha_3}
               \alpha_3:=\frac{1}{\gamma_{3+}}+\frac{\overline c_2}{\beta}+\left(1-\frac{\gamma_{3-}}{\gamma_{3+}}\right)\frac{\alpha_{3-}}{1-a},  \quad \alpha_{3+}:=\frac{1}{\gamma_{3+}}(1-a-\gamma_{3-}\alpha_{3-}),
            \end{equation}
            and  $\alpha_{3-}$ is the unique solution to $\psi(z)=0$ on $\left(\alpha_{LB},\alpha_{UB}\right)$ with
        \begin{equation}\label{eqn: alphaLB & UB}
                \alpha_{LB}:=\frac{1-a}{\gamma_{3-}}\cdot\mathds{1}_{\left\{\overline c_2\geq \frac{N_3N_2}{2N_1}\right\}}+\alpha_0\cdot\mathds{1}_{\left\{\overline c_2<\frac{N_3N_2}{2N_1}\right\}}
                \quad \text{and} \quad 
                \alpha_{UB}:=\frac{(1-a)(\gamma_{3+}-\gamma_{4-})}{\gamma_{3-}(\gamma_{3+}-\gamma_{3-})}.
        \end{equation}
        The relation $w_0\leq u_1\leq u_2$ holds.
        
            \item The function $g$, defined by
        \begin{equation}\label{eqn:g first}
		g(x)=
		\begin{cases}
			\frac{2\lambda(1-\gamma_1)}{w_0}\left(\frac{x}{w_0}\right)^{\gamma_1}&\mbox{if } x < w_0,\\
			-\lambda\left[\gamma_{2-}e^{\gamma_{2+} (x-w_0)}+\gamma_{2+}e^{\gamma_{2-} (x-w_0)}\right]&\mbox{if $w_0\leq x < u_1$,}\\
			\alpha_{3+}e^{\gamma_{3-}(x-u_1)}+\alpha_{3-}e^{\gamma_{3+}(x-u_1)}+\frac{(1-a)\overline c_2}{\beta}&\mbox{if $u_1\leq x < u_2$,}\\
			\frac{a}{\gamma_{4-}}e^{\gamma_{4-} (x-u_2)}+\frac{a\overline{c}_1+(1-a)\overline{c}_2}{\beta}&\mbox{if $x\geq u_2$,}
		\end{cases}
	\end{equation} 
    where 
    \begin{align}\label{eqn:lambda}
        \lambda=-\frac{1-a}{\gamma_{2+}\gamma_{2-}} \left[e^{\gamma_{2+}(u_1-w_0)}+e^{\gamma_{2-}(u_1-w_0)} \right]^{-1},
    \end{align}
    is a classical solution to the HJB equation in \eqref{eqn:hjborig} and thus equals the value function $V$ of the optimization problem in \eqref{eq:V}. In addition, $g$ is strictly concave as hypothesized. 
        
            \item The optimal reinsurance and dividend strategies $(\theta_1^*, \theta_2^*, C_1^*, C_2^*)$ are given by
            \begin{equation*}
			(\theta_1^*, \theta_2^*, C_1^*, C_2^*) (x)=
			\begin{cases}
				\left(1-\frac{x}{w_0},1-\frac{x}{w_2},0,0\right)&\mbox{if $x < w_0$,}\\
				\left(0,1-\frac{w_0}{w_2},0,0\right)&\mbox{if $w_0\leq x<u_1$,}\\
				\left(0,1-\frac{w_0}{w_2},0,\overline{c}_2\right)&\mbox{if $u_1\leq x<u_2$,}\\
				\left(0,1-\frac{w_0}{w_2},\overline{c}_1,\overline{c}_2\right)&\mbox{if $x\geq u_2$.}
			\end{cases}
		\end{equation*}
        \end{enumerate}		
	\end{theorem}
    
    \begin{remark}\label{remark: theorem 1}
        We first highlight the significance of the function $\psi$ in \eqref{eqn:psi}. Together with the definition of $u_2$ in \eqref{eqn:defn of u1 & u2}, the (unique) root of $\psi$ ensures that $g$ is differentiable at $x=u_2$. This root, which is given by $\alpha_{3-}$, can be found in the interval $\left(\frac{1-a}{\gamma_{3-}},\alpha_{UB}\right)$, which is proved in Lemma \ref{lemma:exist alpha3}. 

        Next, we discuss the optimal reinsurance and dividend strategies $(\theta_1^*, \theta_2^*, C_1^*, C_2^*)$ obtained in Item 3 of Theorem \ref{prop:w0<u1<u2}. Recall that we set $a \le 1/2$, implying that the interest of Line 2 outweighs that of Line 1 in the manager's decision. When the aggregate surplus $x$ is low ($x < u_1$), neither line distributes dividends; when $x$ exceeds the first threshold $u_1$, Line 2 starts to pay dividends, but Line 1 waits until $x$ grows beyond the second threshold $u_2 > u_1$. The relative importance of Line 2 over Line 1 is also reflected in the fact that the manager cedes at least $1 - \frac{w_0}{w_2}$ proportion of Line 2's risk to the reinsurer, but for larger enough surplus ($x \ge w_0$), the entire risk of Line 1 is retained.

        Note that the value function $V = g$ in \eqref{eqn:g first} has three switching points, $w_0$, $u_1$, and $u_2$. In this case, to ensure that $V$ is a classical solution, we encounter three systems of equations, with each system ensuring that $V$ and its first and second derivatives are continuous at the corresponding switching point (also referred to as “smooth fit” conditions). This introduces a significant increase in complexity when determining the unknown variables. Notably, one variable (i.e. $\alpha_{3-}$) remains implicitly defined and cannot be explicitly obtained; in comparison,  all unknowns in \cite{hojgaard1999} (which solves a univariate optimal dividend problem) can be explicitly determined. To overcome this technical issue, we first identify the interval containing this unknown variable, ensuring that (i) the formulas of the switching points are well-defined and satisfy specific ordering conditions, and (ii) the value function is increasing. We then prove that this unknown variable is the (unique) root of a monotonic function.

    \end{remark}

In the second case, we assume that Condition $(ii)$ in Theorem \ref{prop:w0<u1<u2} does not hold (that is, $\psi(\alpha_0) > 0$), but Condition $(i)$ still holds. We define a new function, $\chi :\mathbb{R} \mapsto [0,\infty)$ by
    \begin{equation}\label{eqn:xsoln}
		\chi(z)=k_1e^{\frac{N_2}{N_1}z}+\frac{\overline{c}_2(N_2-N_1)}{N_2\beta}z+k_2,
	\end{equation}
    where $k_1$ and $k_2$ are two constants defined case by case. We present the results for the second case in the theorem below.    
    
    \begin{theorem}\label{prop:u1<w0<u2}
		Let Assumption \ref{assume:w1<w2} hold. Suppose (i) $\overline c_1+\overline c_2\geq \frac{\beta w_1}{\gamma_1(1-\gamma_1)}=\frac{N_3N_2}{2N_1}$ and (ii) $\psi(\alpha_{0})> 0$, where $\alpha_{0}$ is defined in \eqref{eqn:alpha_0}. We have the following results:
        \begin{enumerate}
            \item $u_1$ and $u_2$ defined in \eqref{eqn:defn of u1 & u2} are explicitly given by 
            \begin{align}
                u_1&=k_1(1-a)^{-\frac{N_2}{N_1}}-\frac{\overline c_2(N_2-N_1)}{N_2\beta}\ln (1-a)+k_2\\
                 \text{and} \quad    u_2&=u_1+\frac{1}{\gamma_{3+}-\gamma_{3-}}\ln\left(\frac{\alpha_{3-}\gamma_{3-}(\gamma_{4-}-\gamma_{3-})}{\alpha_{3+}\gamma_{3+}(\gamma_{3+}-\gamma_{4-})}\right),
            \end{align}
            where 
            \begin{equation}
            \label{eqn:k}
            \begin{aligned}
                k_1&=\frac{N_1(N_2-N_1)}{N_2\beta}\left(\frac{N_3}{2N_1}-\frac{\overline c_2}{N_2}\right)\left[\gamma_{3+}\alpha_{3+}e^{\gamma_{3+}(w_0-u_1)}+\gamma_{3-}\alpha_{3-}e^{\gamma_{3-}(w_0-u_1)}\right]^{\frac{N_2}{N_1}},\\
                k_2&=\frac{\overline c_2(N_2-N_1)}{N_2\beta}\left(\frac{N_1}{N_2}+\ln(1-a)\right),
            \end{aligned}
        \end{equation}
        $\alpha_{3+}$ is defined in \eqref{eqn:alpha_3}, and $\alpha_{3-}$ is the unique solution to $\psi(z)=0$ on $\left(\frac{1-a}{\gamma_{3-}},\alpha_{0}\right)$.
        
         The zero of $\widehat{\theta}_1$ in \eqref{eqn:defn of w0} is given by 
            \begin{align}
                w_0 = u_1+\frac{1}{\gamma_{3+}-\gamma_{3-}}\ln\left[\frac{\gamma_{3-}\alpha_{3-}\left(\frac{w_1}{\gamma_1-1}\gamma_{3-}-1\right)}{\gamma_{3+}\alpha_{3+}\left(1-\frac{w_1}{\gamma_1-1}\gamma_{3+}\right)}\right],
            \end{align}
            with $u_1$ obtained above. 
            
            In addition, the relation $u_1<w_0\leq u_2$ holds.
            
            \item The  function $g$, defined by
        \begin{equation}\label{eqn:candidate 3}
		g(x)=
		\begin{cases}
			\frac{(1-a)u_1}{\gamma_1}\left(\frac{x}{u_1}\right)^{\gamma_1}&\mbox{if $x<u_1$,}\\
			\int_{u_1}^x e^{- \chi^{-1}(y)} \dd y+\frac{(1-a)u_1}{\gamma_1}&\mbox{if $u_1\leq x<w_0$,}\\
			\alpha_{3+}e^{\gamma_{3-}(x-u_1)}+\alpha_{3-}e^{\gamma_{3+}(x-u_1)}+\frac{(1-a)\overline c_2}{\beta}&\mbox{if $w_0\leq x<u_2$,}\\
			\frac{a}{\gamma_{4-}}e^{\gamma_{4-} (x-u_2)}+\frac{a\overline{c}_1+(1-a)\overline{c}_2}{\beta}&\mbox{if $x\geq u_2$,}
		\end{cases}
	\end{equation} 
    where $\chi^{-1}$ denotes the inverse function of $\chi$ in \eqref{eqn:xsoln}, 
    is a classical solution to the HJB equation in \eqref{eqn:hjborig} and thus equals the value function $V$ of the optimization problem in \eqref{eq:V}. In addition, $g$ is strictly concave.
   
            \item The reinsurance and dividend strategies $(\theta_1^*, \theta_2^*, C_1^*, C_2^*)$ are given by
            \begin{equation*}
			(\theta_1^*, \theta_2^*, C_1^*, C_2^*) (x) = 
			\begin{cases}
				\left(1-\frac{x}{w_1},1-\frac{x}{w_2},0,0\right)&\mbox{if $x<u_1$,}\\
				\left(1-\frac{1-\gamma_1}{w_1} \chi'(\chi^{-1}(x)),1-\frac{1-\gamma_1}{w_2} \chi'(\chi^{-1}(x)),0,0\right)&\mbox{if $u_1\leq x<w_0$,}\\
				\left(0,1-\frac{1-\gamma_1}{w_2} \chi'(\chi^{-1}(w_0)),0,\overline{c}_2\right)&\mbox{if $w_0\leq x<u_2$,}\\
				\left(0,1-\frac{1-\gamma_1}{w_2} \chi'(\chi^{-1}(w_0)),\overline{c}_1,\overline{c}_2\right)&\mbox{if $x\geq u_2$.}
			\end{cases}
		\end{equation*}
        \end{enumerate}		
	\end{theorem}

    \begin{remark}\label{remark: theorem 2}
        By comparing Theorem \ref{prop:u1<w0<u2} with Theorem \ref{prop:w0<u1<u2}, we observe that the manager's optimal strategies $(\theta_1^*, \theta_2^*, C_1^*, C_2^*)$ are similar in these two cases, and thus the explanations in Remark \ref{remark: theorem 1} apply to Theorem \ref{prop:u1<w0<u2} as well. However, there are differences on the thresholds $u_1 < w_0 < u_2$ and the exact form of the optimal ceded proportion $\theta_i^*$. 

        Under the conditions of Theorem \ref{prop:u1<w0<u2}, the root of $\psi$ in  \eqref{eqn:psi} is found over the interval $((1-a)/\gamma_{3-},\alpha_0)$. It is important to note that, in this case, the value function involves an integral that cannot be expressed in closed form, which requires numerical approximations to solve the integral. A computationally efficient approach to solving this integral is discussed in Section \ref{sec numerical}. 
    \end{remark}

Last, we consider the case when Condition $(i)$ of Theorems \ref{prop:w0<u1<u2} and \ref{prop:u1<w0<u2} does not hold. The results under this case are summarized below. 
    
    \begin{theorem}\label{prop:u1<u2<w0}
		Let Assumption \ref{assume:w1<w2} hold. Suppose $\overline c_1+\overline c_2<\frac{\beta w_1}{\gamma_1(1-\gamma_1)}=\frac{N_3N_2}{2N_1}$. We have the following results:
        \begin{enumerate}
            \item $w_0$ defined in \eqref{eqn:defn of w0} is infinite ($w_0 = \infty$), and $u_1$ and $u_2$ defined in \eqref{eqn:defn of u1 & u2} are explicitly given by 
            \begin{align}
                u_1&=(1-\gamma_1)\left[\frac{\overline c_1(N_2-N_1)}{N_2\beta}\left(\frac{a}{1-a}\right)^{\frac{N_2}{N_1}}+\frac{\overline c_2(N_2-N_1)}{N_2\beta}\right]\\
                \text{and} \quad u_2&=u_1+\frac{\overline c_1N_1(N_2-N_1)}{N_2^2\beta}\left(1-\left(\frac{a}{1-a}\right)^{\frac{N_2}{N_1}}\right)-\frac{\overline c_2(N_2-N_1)}{N_2\beta}\ln \left(\frac{a}{1-a}\right),
            \end{align}
            respectively. Moreover, the relation $u_1\leq u_2<w_0$ holds.

            \item The function $g$, defined by
        \begin{equation}\label{eqn:g w0>u1}
		g(x)=
		\begin{cases}
			\frac{(1-a)u_1}{\gamma_1}\left(\frac{x}{u_1}\right)^{\gamma_1}&\mbox{if $x<u_1$,}\\
			\int_{u_1}^xe^{-\chi^{-1}(y)} \dd y+\frac{(1-a)u_1}{\gamma_1}&\mbox{if $u_1\leq x<u_2$,}\\
			\frac{a}{\gamma_3}e^{\gamma_3 (x-u_2)}+\frac{a\overline c_1+(1-a)\overline c_2}{\beta}&\mbox{if $x\geq u_2$,}
		\end{cases}
	\end{equation}  where 
    $\chi^{-1}$ is the inverse of the function $\chi$ in \eqref{eqn:xsoln} with
            \begin{equation}\label{eqn:k1 & k2 case 3}
		\begin{aligned}
			k_1&=\frac{u_2-u_1+\frac{\overline{c}_2(N_2-N_1)}{N_2\beta}\ln \left(\frac{a}{1-a}\right)}{a^{-\frac{N_2}{N_1}}-(1-a)^{-\frac{N_2}{N_1}}},\\
			k_2&=\frac{a^{-\frac{N_2}{N_1}}\left(u_1+\frac{\overline{c}_2(N_2-N_1)}{N_2\beta}\ln (1-a)\right)-(1-a)^{-\frac{N_2}{N_1}}\left(u_2+\frac{\overline{c}_2(N_2-N_1)}{N_2\beta}\ln a\right)}{a^{-\frac{N_2}{N_1}}-(1-a)^{-\frac{N_2}{N_1}}},
		\end{aligned}
	\end{equation}
        and
	\begin{equation}\label{eqn:gamma3}
		\gamma_3=-\frac{N_2\beta}{(\overline c_1+\overline c_2)(N_2-N_1)}<0,
	\end{equation}
    is a classical solution to the HJB equation in \eqref{eqn:hjborig} and thus equals the value function $V$ of the optimization problem in \eqref{eq:V}. In addition, $g$ is strictly concave.

            \item The reinsurance and dividend strategies $(\theta_1^*, \theta_2^*, C_1^*, C_2^*)$ are given by
            \begin{equation*}
			(\theta_1^*, \theta_2^*, C_1^*, C_2^*)(x) =
			\begin{cases}
				\left(1-\frac{x}{w_1},1-\frac{x}{w_2},0,0\right)&\mbox{if $x<u_1$,}\\
				\left(1-\frac{1-\gamma_1}{w_1}\chi'(\chi^{-1}(x)),1-\frac{1-\gamma_1}{w_2}\chi'(\chi^{-1}(x)),0,\overline{c}_2\right)&\mbox{if $u_1\leq x<u_2$,}\\
				\left(1+\frac{1-\gamma_1}{w_1\gamma_3},1+\frac{1-\gamma_1}{w_2\gamma_3},\overline{c}_1,\overline{c}_2\right)&\mbox{if $x\geq u_2$.}
			\end{cases}
		\end{equation*}
        In this case, $\theta_i^* >0$ for $i=1,2$.
        \end{enumerate}		
	\end{theorem}
    
    \begin{remark}
        Under the conditions of Theorem \ref{prop:u1<u2<w0}, the manager is only allowed to pay dividends up to $\frac{N_3N_2}{2N_1}$ for two lines together. This frees up the aggregate reserve, allowing the manager to purchase reinsurance coverage for both lines (with $\theta_i^* > 0$) in regardless of the surplus level $x$. When $x$ is relatively small, $\theta_i^*$ decreases with respect to $x$; however, for large enough $x$ (when $x \ge u_2$), the reinsurance decision is \emph{independent} of $x$, and there is a maximum ceded proportion for each line.   

        In the above three main theorems, we find that the value function in each scenario (see \eqref{eqn:g first}, \eqref{eqn:candidate 3}, and \eqref{eqn:g w0>u1}) features at least two, but at most three, switching points. One switching point serves as a signal for the manager to stop purchasing more reinsurance, while the other two switching points signal the manager to pay out the maximum dividend rate. This also extends the results in \cite{hojgaard1999}, where they have at least one, but at most two, switching points. Notably, they show that it is not possible for the switching point that affects reinsurance to be higher than the one that governs the dividend payout. However, such a case is possible within our framework. 
    \end{remark}

    To derive the optimal capital injection strategy,  we partition the domain of the surplus pair $(x_1, x_2) \in \Rb_+^2$ into 7 regions (see Figure \ref{fig:capitaltransfer}). Recall that in each of Theorems \ref{prop:w0<u1<u2}, \ref{prop:u1<w0<u2}, and \ref{prop:u1<u2<w0}, we obtain $w_0$, $u_1$, and $u_2$ explicitly.  With that in mind, define constants $\delta_i$, $i=1,2,3$, corresponding to each of the three cases by 
     \begin{equation}
        (\delta_0,\delta_1,\delta_2)=
        \begin{cases}
            (w_0,u_1,u_2)&\mbox{if $w_0\leq u_1\leq u_2$,}\\
            (u_1,w_0,u_2)&\mbox{if $u_1< w_0\leq u_2$,}\\
            (u_1,u_1,u_2)&\mbox{if $u_1\leq u_2<w_0$}.
        \end{cases}
    \end{equation}
    The 7 regions $A_i$, $i=1,2,\cdots, 7$, are defined as follows (see Figure \ref{fig:capitaltransfer}):
        \begin{itemize}
            \item $A_1=\{(x_1,x_2):x_1\geq 0, x_2>\delta_2\},$
            \item $A_2=\{(x_1,x_2):x_1>0, x_2\in [0,\delta_2], x_1+x_2> \delta_2\},$
            \item $A_3=\{(x_1,x_2):x_1\geq0, x_2\in (\delta_1,\delta_2], x_1+x_2\leq \delta_2\},$
            \item $A_4=\{(x_1,x_2):x_1> 0, x_2\in [0,\delta_1], x_1+x_2\in (\delta_1,\delta_2]\},$
            \item $A_5=\{(x_1,x_2):x_1\geq 0, x_2\in (\delta_0,\delta_1], x_1+x_2\leq \delta_1\},$
            \item $A_6=\{(x_1,x_2):x_1> 0, x_2\in [0,\delta_0], x_1+x_2\in (\delta_0,\delta_1]\},$
            \item $A_7=\{(x_1,x_2):x_1\geq 0, x_2\geq 0, x_1+x_2\leq \delta_0\}.$
        \end{itemize}
   
    \begin{figure}[htb]
			\centering
			\begin{tikzpicture}[scale=1] 
				\draw[->] (-1,0) -- (7,0) node[right] {\(x_1\)};
				\draw[->] (0,-1) -- (0,7) node[above] {\(x_2\)};
				
				\draw (0,6) node[left] {\((0,\delta_2)\)} -- (6,0) node[below] {\((\delta_2,0)\)}; 
				\draw (0,4) node[left] {\((0,\delta_1)\)} -- (4,0) node[below] {\((\delta_1,0)\)};
				\draw (0,2) node[left] {\((0,\delta_0)\)} -- (2,0) node[below] {\((\delta_0,0)\)};
				\draw (0,6) -- (7,6) node[right] {\(x_2 = \delta_2\)}; 
				\draw (0,4) -- (2,4) node[right] {}; 
				\draw (0,2) -- (2,2) node[right] {}; 
				
				\filldraw[black] (0,4) circle (2pt); 
				\filldraw[black] (0,6) circle (2pt); 
				\filldraw[black] (6,0) circle (2pt);
				\filldraw[black] (0,2) circle (2pt);
				\filldraw[black] (2,0) circle (2pt);
				\filldraw[black] (4,0) circle (2pt);
				
				\node at (0.5, 0.9) {$A_7$};
                \node at (2.1, 0.9) {$A_6$};
                \node at (3, 2) {$A_4$};
                \node at (0.5, 2.9) {$A_5$};
                \node at (4.1, 2.9) {$A_2$};
                \node at (0.5, 4.9) {$A_3$};
                \node at (0.5, 6.5) {$A_1$};

				\node[rotate=-45] at (1.1, 1.1) {\footnotesize\(x_1+x_2 = \delta_0\)};
                \node[rotate=-45] at (3.1, 1.1) {\footnotesize\(x_1+x_2 = \delta_1\)};
                \node[rotate=-45] at (5.1, 1.1) {\footnotesize\(x_1+x_2 = \delta_2\)};
                \node at (1, 4.2) {\footnotesize\(x_2 = \delta_1\)};
                \node at (1, 2.2) {\footnotesize\(x_2 = \delta_0\)};
			\end{tikzpicture}
			\caption{Regions for Capital Injection Decisions}
			\label{fig:capitaltransfer}
		\end{figure}
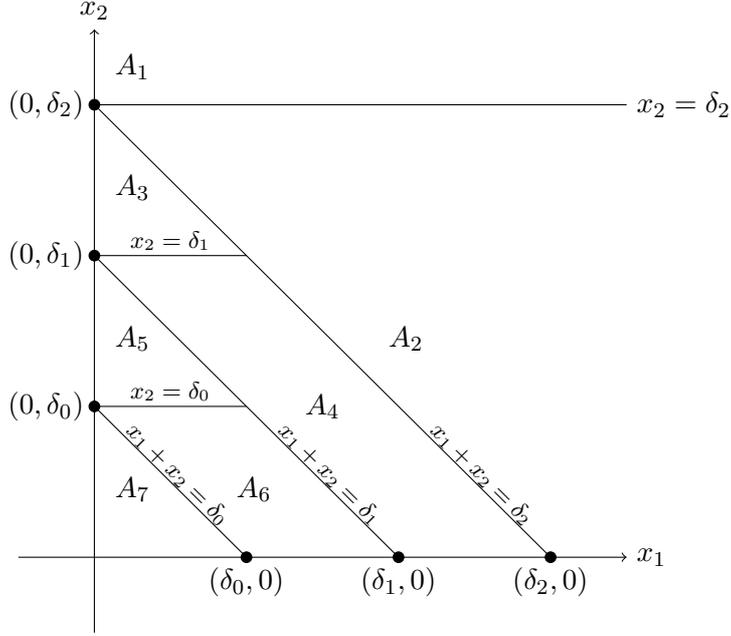

\begin{theorem}
    \label{thm:op_L}
The optimal capital injection strategy is given by one of the following cases:
\begin{enumerate}

            \item If $x\in A_1$ and Line 1 hits zero, the manager transfers an amount of $x_2-\delta_2$ from Line 2 to Line 1, and we proceed to region $A_{2}$. If Line 1 does not hit zero, we stay in $A_1$ until we move to region $A_2$ or $A_3$.

            \item If $x\in A_2$ and Line 2 hits zero, the manager transfers an amount of $x_1-\delta_2$ from Line 1 to Line 2. We stay in $A_2$ until we move to region $A_1$, $A_3$, or $A_4$, regardless of whether Line 2 hits zero.

            \item If $x\in A_3$ and Line 1 hits zero, the manager transfers an amount of $x_2-\delta_1$ from Line 2 to Line 1, and we proceed to region $A_4$. If Line 1 does not hit zero, we stay in $A_3$ until we move to region $A_2$, $A_4$, or $A_5$.

            \item If $x\in A_4$ and Line 2 hits zero, the manager transfers an amount of $x_1-\delta_1$ from Line 1 to Line 2. We stay in $A_4$ until we move to region $A_2$, $A_3$, $A_5$, or $A_6$, regardless of whether Line 2 hits zero.

            \item If $x\in A_5$ and Line 1 hits zero, the manager transfers an amount of $x_2-\delta_0$ from Line 2 to Line 1, and we proceed to region $A_6$. If Line 1 does not hit zero, we stay in $A_5$ until we move to region $A_4$, $A_6$, or $A_7$.

            \item If $x\in A_6$ and Line 2 hits zero, the manager transfers an amount of $x_1-\delta_0$ from Line 1 to Line 2. We stay in $A_6$ until we move to region $A_4$, $A_5$, or $A_7$, regardless of whether Line 2 hits zero.

            \item If $x\in A_7$, we stay in $A_7$ until we move to region $A_6$. The problem ends when the surplus pair leaves the positive quadrant.
		\end{enumerate}	
\end{theorem}

\subsection{Remaining cases}
\label{sub:remain}

In this section, we discuss the remaining cases that are not covered in Section \ref{main subsection}. Since the analysis and technical proofs are similar to those in Section \ref{main subsection}, we directly present the results. Recall that we assume $0 <\rho < \frac{\mu_1 \slash \mu_2}{\sigma_1\slash \sigma_2} < \frac{1}{\rho}$ in Section \ref{main subsection} (see \eqref{eq:rho_case1}).

First, we consider the case that 
\begin{align}
\label{eq:rho_case2}
    \rho \ge  \frac{\mu_1 \slash \mu_2}{\sigma_1\slash \sigma_2} > 0.
\end{align}
In this case, we know from \eqref{eqn:optimal_theta} that for the candidate reinsurance strategy for Line 1, $\widehat\theta_1(x)\geq 1$ holds for all $x\geq 0$. As such, with the constraint $\theta_i\in[0,1]$ in place, we have $\theta^*_1(x) = 1$, and Line 1 cedes all of its risk to the reinsurer. Previously in Section \ref{main subsection}, we define $w_0$ as the zero of $\widehat{\theta}_1$ in \eqref{eqn:defn of w0}; now when \eqref{eq:rho_case2} holds, we define $w_0$ as the zero of $\widehat{\theta}_2$ by 
\begin{equation}
\label{eq:w0_case2}
        \widehat\theta_2(w_0)=1+\frac{1-\gamma_1}{w_2}\cdot\frac{g'(w_0)}{g''(w_0)}=0.
    \end{equation}    

\begin{proposition}
\label{prop:case2}
    Assume that the correlation coefficient $\rho$ satisfies \eqref{eq:rho_case2}. We define $w_0$ by \eqref{eq:w0_case2} and constants $N_i$ by 
    \begin{align*}
     N_1=\mu_2^2, \quad N_2=\mu_2^2+2\beta\sigma_2^2, \quad N_3=\mu_2, \quad \text{and} \quad N_4=\sigma_2^2,
    \end{align*}
    and the rest of the notation follows from those defined in Section \ref{main subsection}. The manager's optimal reinsurance strategy $(\theta_1^*, \theta_2^*)$ is given by one of the following scenarios:
       \begin{enumerate}
        \item If $\overline c_1+\overline c_2\geq \frac{\beta w_0}{\gamma_1(1-\gamma_1)}=\frac{N_3N_2}{2N_1}=\frac{\mu_2}{2}+\frac{\sigma_2^2\beta}{\mu_2}$ and $\psi(\alpha_0)\leq 0$, then
        \begin{equation}
            (\theta_1^*, \theta_2^*)(x) =
            \begin{cases}
                \left(1,1-\frac{x}{w_0}\right)&\mbox{if $x<w_0$,}\\
                \left(1,0\right)&\mbox{if $x \ge w_0$}.
            \end{cases}
        \end{equation}
        \item If $\overline c_1+\overline c_2\geq \frac{\mu_2}{2}+\frac{\sigma_2^2\beta}{\mu_2}$ and $\psi(\alpha_0)> 0$, then
        \begin{equation}
            (\theta_1^*, \theta_2^*)(x) =
            \begin{cases}
                \left(1,1-\frac{x}{w_2}\right)&\mbox{if $x<u_1$,}\\
                \left(1,1-\frac{1-\gamma_1}{w_2}\chi'(\chi^{-1}(x))\right)&\mbox{if $u_1 \ge x<w_0$,}\\
                \left(1,1-\frac{1-\gamma_1}{w_2}\chi'(\chi^{-1}(w_0))\right)&\mbox{if $x \ge w_0$}.
            \end{cases}
        \end{equation}
        \item If $\overline c_1+\overline c_2< \frac{\mu_2}{2}+\frac{\sigma_2^2\beta}{\mu_2}$, then
        \begin{equation}
            (\theta_1^*, \theta_2^*)(x) =
            \begin{cases}
                \left(1,1-\frac{x}{w_2}\right)&\mbox{if $x<u_1$,}\\
                \left(1,1-\frac{1-\gamma_1}{w_2}\chi'(\chi^{-1}(x))\right)&\mbox{if $u_1 \ge x<u_2$,}\\
                \left(1,1+\frac{1-\gamma_1}{w_2\gamma_3}\right)&\mbox{if $x \ge u_2$}.
            \end{cases}
        \end{equation}
    \end{enumerate}
    Moreover, the candidate strategy defined in \eqref{eqn:optimal dividend} is the optimal dividend strategy $(C_1^*,C_2^*)$, and the optimal capital transfer strategy is the same as the one stated in Theorem \ref{thm:op_L}.
\end{proposition}

	\begin{remark}
		Since Line 1 transfers all of its risk to the reinsurer, the manager focuses on managing Line 2's risk. This reduces the problem to just one business line, which is well studied in the literature. 
        For instance, \cite{hojgaard1999} solve a similar problem, and the condition in their Theorem 2.1 is parallel to $\overline c_1+\overline c_2\geq \frac{\mu_2}{2}+\frac{\sigma_2^2\beta}{\mu_2}$ in the above proposition. 
	\end{remark}

Second, we consider the case of $\frac{\mu_1 \slash \mu_2}{\sigma_1\slash \sigma_2}\geq\frac{1}{\rho} > 0$.
In this case, $\widehat\theta_2(x)=1$ for all $x\geq 0$, and thus $\theta_2^* \equiv 1$, implying that Line 2 transfers all of its risk to the reinsurer. As such, this case is similar to the previous case analyzed in Proposition \ref{prop:case2}. The difference is that we define $w_0$ as in \eqref{eqn:defn of w0} and constants $N_i$ by $N_1=\mu_1^2$, $N_2=\mu_1^2+2\beta\sigma_1^2$, $N_3=\mu_1$, and $N_4=\sigma_1^2$.

Last, we consider the negative correlation case of  $-1 < \rho \le 0$.
In this case, we always have $\widehat\theta_1(x)<1$ and $\widehat\theta_2(x)<1$ for all $x>0$. Therefore, the analysis follows from the one in Section \ref{main subsection} using the same values for $N_1$, $N_2$, $N_3$ and $N_4$ defined in \eqref{eqn:N1N2N3N4}. 

    \begin{remark}
        The ratio $\frac{\mu_i}{\sigma_i}$ measures the trade-off between the mean-adjusted risk exposure, $\mu_i:=\kappa_i\tilde\mu_i$, and the volatility of the risk exposure, $\sigma_i$. Hence, the term $\frac{\mu_1 \slash \mu_2}{\sigma_1\slash \sigma_2}$ can be interpreted as the relative Sharpe ratio of Line 1 over Line 2. A small value for $\frac{\mu_1 \slash \mu_2}{\sigma_1\slash \sigma_2}$ implies a more favorable trade-off for Line 2. This aligns with the case of $0<\frac{\mu_1 \slash \mu_2}{\sigma_1\slash \sigma_2}\leq \rho$ wherein the insurer transfers all of the risk of Line 1 to the reinsurer and focuses solely on Line 2. Conversely, a large value for $\frac{\mu_1 \slash \mu_2}{\sigma_1\slash \sigma_2}$ implies that Line 1 has a more attractive trade-off, justifying the insurer's focus on Line 1 instead. However, in the case of $\rho\leq 0$, this ratio is no longer crucial, because the negative correlation provides a hedging effect or a form of diversification.
    \end{remark}

    \section{Numerical Examples}\label{sec numerical}

    In this section, we conduct a numerical analysis to derive further insights from the theoretical results obtained in Section \ref{sec:sol}. Recall that the manager of the insurer makes three decisions on reinsurance $\theta_i$, dividend payout $C_i$, and capital injection $L_i$ for both lines $i=1,2$. Theorems \ref{prop:w0<u1<u2}, \ref{prop:u1<w0<u2}, and \ref{prop:u1<u2<w0}, along with \ref{thm:op_L}, obtain the optimal dividend and capital injection strategies in \emph{closed form}. As a result, in the numerical analysis, we only focus on the manager's reinsurance decisions, which may take corner solutions ($0$ or $1$) or interior solutions from the first-order condition (FOC). We also plot the value function $V$ ($= g$) for visualization. 

    As we learn from Section \ref{sec:sol}, the case of $0<\rho<\frac{\mu_1\slash\mu_2}{\sigma_1\slash\sigma_2}<\frac{1}{\rho}$ analyzed in Section \ref{main subsection} is the most important one. For this case, there are \emph{three} exclusive scenarios: (1) $w_0\leq u_1 \le u_2$ (referred to as the ``\emph{Main Scenario}") in Theorem \ref{prop:w0<u1<u2}, (2) $u_1<w_0\leq u_2$ in Theorem \ref{prop:u1<w0<u2}, and (iii) $u_1 \le u_2 <w_0$ in Theorem \ref{prop:u1<u2<w0}.
    In addition, we consider the main scenario for the cases of $0<\frac{1}{\rho}\leq\frac{\mu_1\slash\mu_2}{\sigma_1\slash\sigma_2}$ and $\rho\leq 0$ in Section \ref{sub:remain}.

    First, it is important to note that the integral term in the value function $V$ in \eqref{eqn:candidate 3} and \eqref{eqn:g w0>u1} is computationally heavy when evaluated directly, due to the integrand being the exponential of an inverse function. We can reduce this computational burden by introducing a change of variables. More precisely, we let $z=\chi^{-1}(y)$. Then,
    \begin{equation}
        \begin{aligned}
            \int_{u_1}^{x}e^{-\chi^{-1}(y)} \dd y
            &=\int_{-\ln(1-a)}^{\chi^{-1}(x)}e^{-z}\chi'(z) \dd z\\
            &=xe^{-\chi^{-1}(x)}-u_1(1-a)+\int_{-\ln(1-a)}^{\chi^{-1}(x)}e^{-z}\chi(z) \dd z\\
            &=xe^{-\chi^{-1}(x)}-u_1(1-a)+\int_{-\ln(1-a)}^{\chi^{-1}(x)}\left[k_1e^{\left(\frac{N_2}{N_1}-1\right)z}+\frac{\overline c_2(N_2-N_1)}{N_2\beta}ze^{-z}+k_2e^{-z}\right] \dd z\\
            &=\left[\frac{k_1N_1}{N_2-N_1}e^{\frac{N_2}{N_1}\chi^{-1}(x)}-\frac{\overline c_2(N_2-N_1)}{N_2\beta}\left(1+\chi^{-1}(x)\right)+x-k_2\right]e^{-\chi^{-1}(x)}\\
            &\quad-(1-a)\left[\frac{k_1N_1}{N_2-N_1}(1-a)^{-\frac{N_2}{N_1}}-\frac{\overline c_2(N_2-N_1)}{N_2\beta}(1-\ln(1-a))+u_1-k_2\right].
        \end{aligned}
    \end{equation}
    This transformation significantly reduces the computational effort required to evaluate the integral, as we only need to compute the inverse of $\chi$ once for each $x$.

    In Figures \ref{fig:row1} to \ref{fig:row5}, we plot the value function $V$ in the left panel and the optimal reinsurance strategies $\theta_1^*$ and $\theta_2^*$ in the right panel. Note that the dotted vertical lines represent the switching points, $u_1$, $u_2$, and $w_0$. By the definition in  \eqref{eqn:defn of u1 & u2}, $u_1$ (resp., $u_2$) is the point at which the slope of $V$ equals $1-a$ (resp., $a$); $w_0$ serves as the threshold point for the manager to stop purchasing reinsurance and bear all of the risk for Line 1 once the aggregate reserve level exceeds $w_0$ (see its definition in \eqref{eqn:defn of w0}.

    To study the case of $0<\rho<\frac{\mu_1\slash\mu_2}{\sigma_1\slash\sigma_2}<\frac{1}{\rho}$, we fix $\mu_1 = 4$, $\mu_2 = 2$, $\sigma_1 = 1.5$, $\sigma_2 = 1$, and $\rho = 0.6$, along with $\beta = 0.5$ and $a = 0.3$. However, the maximum dividend rates $\oc_1$ and $\oc_2$ may vary.
    When $\oc_1 = 3$ and $\oc_2 = 2$, we compute $w_0=0.58 < u_1=0.62 < u_2=1.49$, which corresponds to the scenario of Theorem \ref{prop:w0<u1<u2}, and the results are plotted in Figure \ref{fig:row1}. 
    Similarly, for a different set of $\oc_1$ and $\oc_2$, Figure \ref{fig:row2} (resp., Figure \ref{fig:row3}) corresponds to the scenario of Theorem \ref{prop:u1<w0<u2} (resp., Theorem \ref{prop:u1<u2<w0}). 
    Overall, there is a decreasing relation between the optimal reinsurance strategy $\theta_i^*$ and  the insurer's aggregate surplus level $x$. In all three scenarios, there exists a threshold value beyond which $\theta_i^*$ is flat and reaches its minimum. Such a minimum level is zero in Figures \ref{fig:row1} and \ref{fig:row2}, but is strictly positive in Figure \ref{fig:row3}. That is, when the zero point $w_0$ does not exist (infinity) as in Figure \ref{fig:row3}, the manager transfers some portion of the risk to the reinsurer for both lines. 

    For the case of $0<\frac{\mu_1\slash\mu_2}{\sigma_1\slash\sigma_1}\leq \rho$ , we present the graphs in Figure \ref{fig:row4}. In this case, $\theta_1^* \equiv 1$, and the manager cedes all the risk of Line 1 to the reinsurer, in regardless of the aggregate surplus $x$. But for Line 2, the manager quickly reduces the ceded proportion to 0 when $x$ exceeds the switching point $w_0$. 
    The last case we consider is $-1 < \rho \le 0$, and the results are plotted in Figure \ref{fig:row5}. In this case, $\theta_1^*(x) \ge \theta_2^*(x)$ for all $x$, and $\theta_2^*$ decreases to 0 rapidly.

    It is clear from all figures that the value function $V$ is strictly increasing and strictly concave. Also recall that $V$ approaches the limit $\frac{a\overline c_1+(1-a)\overline c_2}{\beta}$ when the aggregate surplus $x$ increases to infinity. Taking $V$ in Figure \ref{fig:row1} as the ``benchmark," we observe that smaller values of $u_1$ and $u_2$ result in $V$ reaching its limit more rapidly, as seen in Figure \ref{fig:row5}. In contrast, larger values of $u_1$ and $u_2$ lead to a slower convergence to the limit, as shown in Figure \ref{fig:row4}.

    \begin{figure}[h]
    \centering
    \begin{subfigure}[H]{0.45\textwidth}
        \centering
        \includegraphics[width=\linewidth, trim = 0cm 0.5cm 1cm 2cm, clip = true]{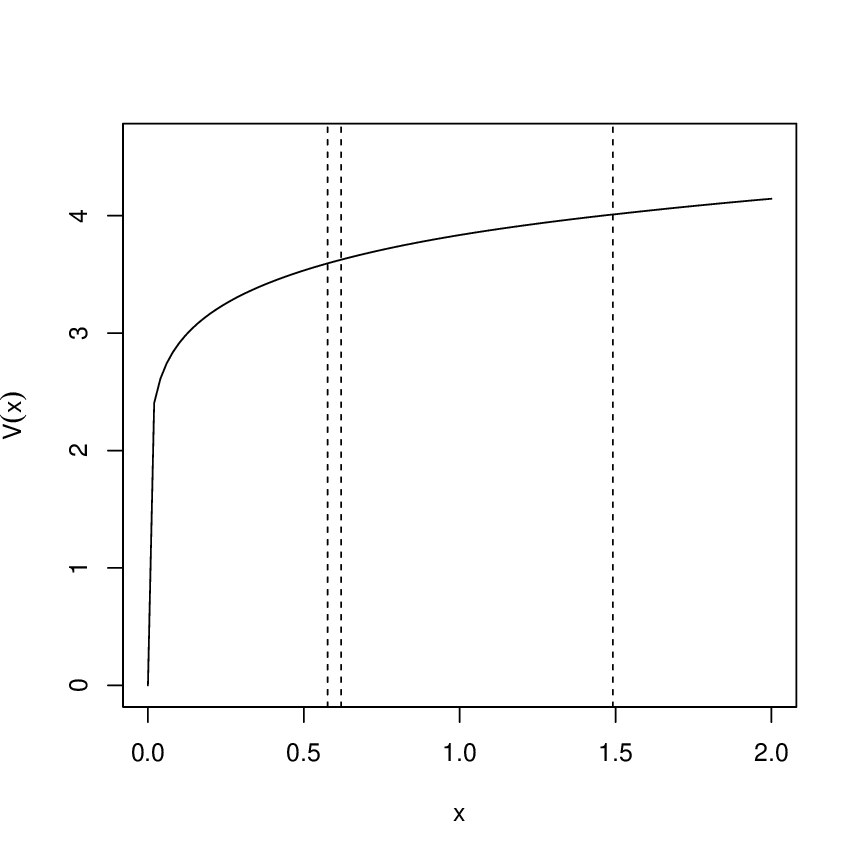} 
    \end{subfigure}
    \hfill
    \begin{subfigure}[H]{0.45\textwidth}
        \centering
        \includegraphics[width=\linewidth, trim = 0cm 0.5cm 1cm 2cm, clip = true]{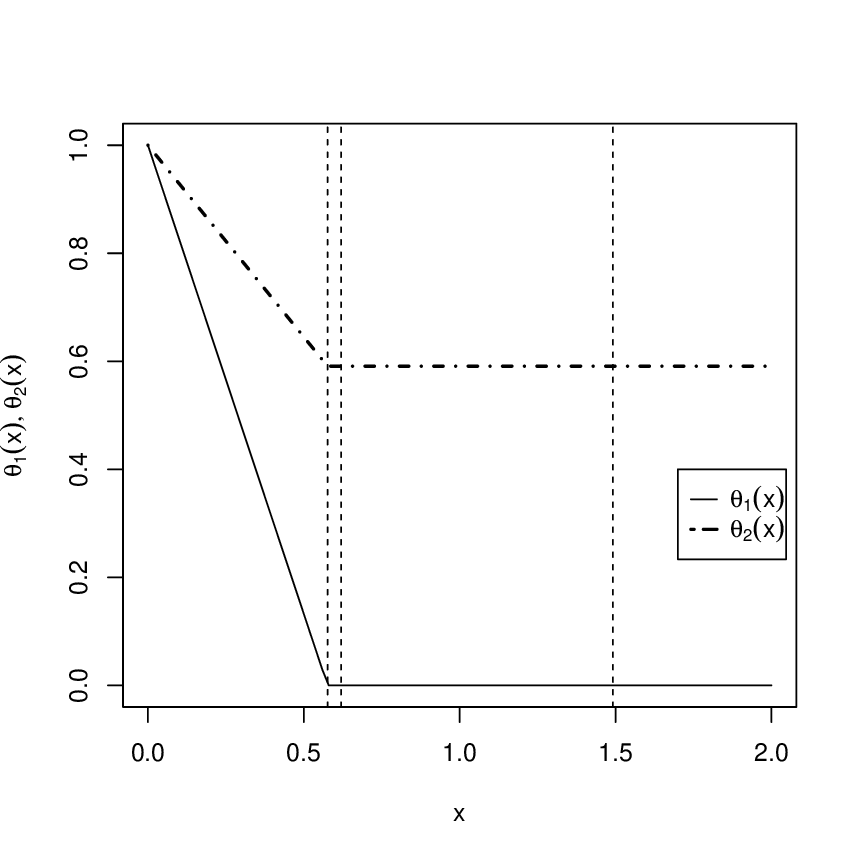}
    \end{subfigure}
    \\[-2ex]
    \caption{$\mu_1=4$, $\mu_2=2$, $\sigma_1=1.5$, $\sigma_2=1$, $\rho=0.6$, $\beta =0.5$, $a=0.3$, $\overline c_1=3$, $\overline c_2=2$ \\($w_0=0.58 < u_1=0.62 < u_2=1.49$, corresponding to Theorem \ref{prop:w0<u1<u2})}
    \label{fig:row1}
    \end{figure}

    \begin{figure}[h]
    \centering
    \begin{subfigure}[H]{0.45\textwidth}
        \centering
        \includegraphics[width=\linewidth, trim = 0cm 0.5cm 1cm 2cm, clip = true]{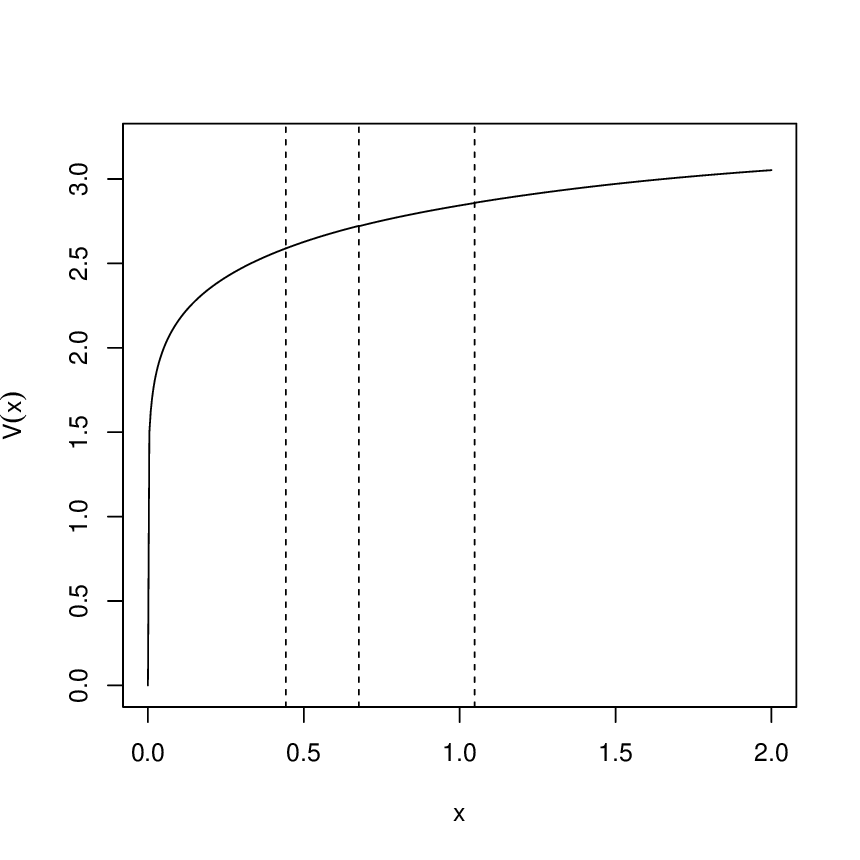} 
        \label{fig:plot3}
    \end{subfigure}
    \hfill
    \begin{subfigure}[H]{0.45\textwidth}
        \centering
        \includegraphics[width=\linewidth, trim = 0cm 0.5cm 1cm 2cm, clip = true]{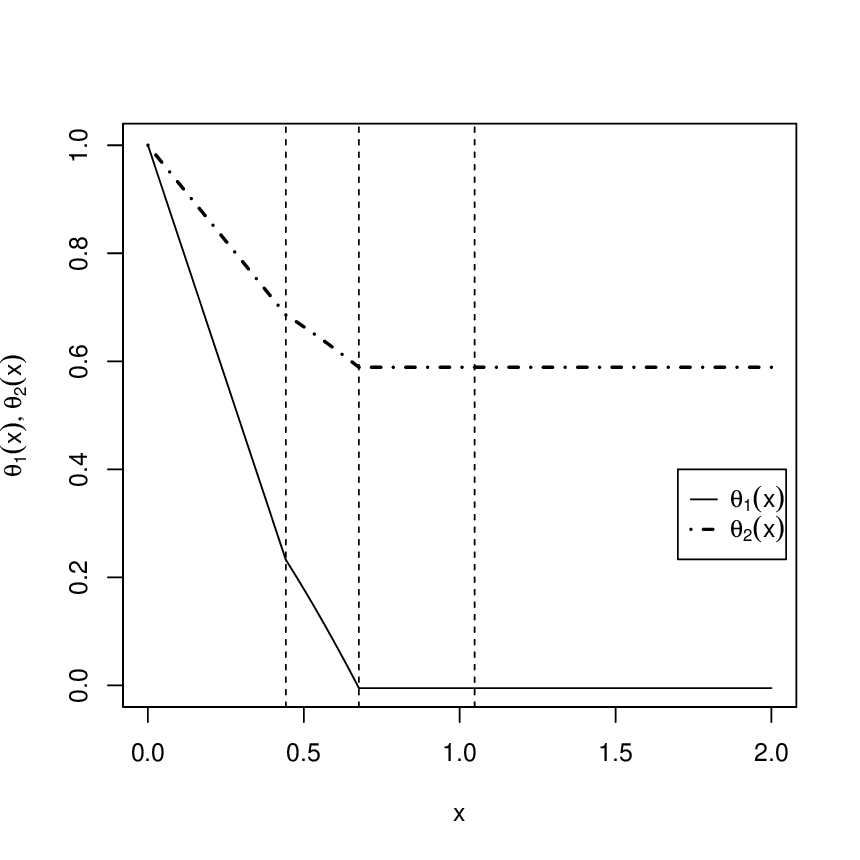}
        \label{fig:plot4}
    \end{subfigure}
    \\[-4ex]
   \caption{$\mu_1=4$, $\mu_2=2$, $\sigma_1=1.5$, $\sigma_2=1$, $\rho=0.6$, $\beta =0.5$, $a=0.3$, $\overline c_1=3$, $\overline c_2=1$ \\($u_1=0.44 < w_0=0.68 < u_2=1.05$, corresponding to Theorem \ref{prop:u1<w0<u2})}
    \label{fig:row2}
    \end{figure}
    
    \begin{figure}[H]
    \begin{subfigure}[H]{0.45\textwidth}
        \centering
        \includegraphics[width=\linewidth, trim = 0cm 0.5cm 1cm 2cm, clip = true]{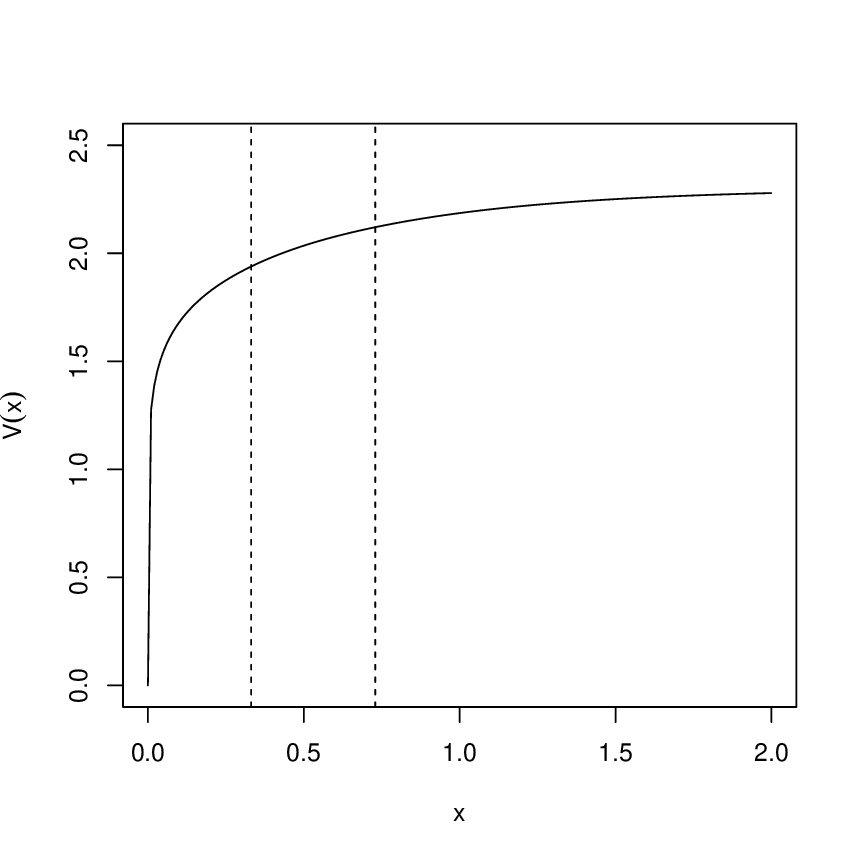}
        \label{fig:plot5}
    \end{subfigure}
    \hfill
    \begin{subfigure}[H]{0.45\textwidth}
        \centering
        \includegraphics[width=\linewidth, trim = 0cm 0.5cm 1cm 2cm, clip = true]{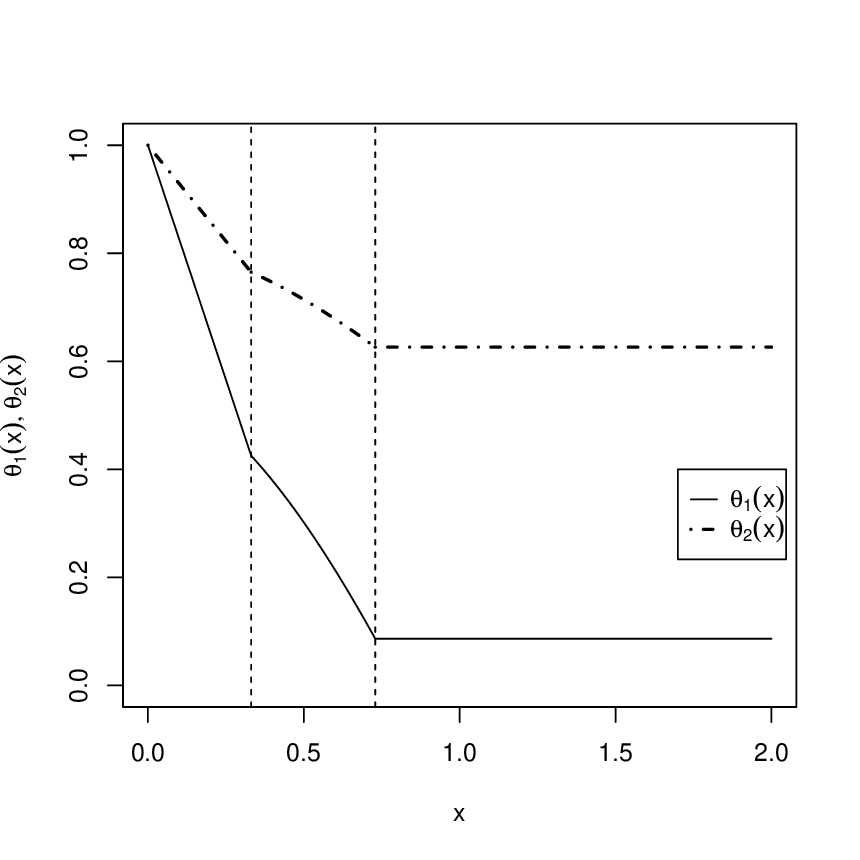}
        \label{fig:plot6}
    \end{subfigure}
    \\[-4ex]
   \caption{$\mu_1=4$, $\mu_2=2$, $\sigma_1=1.5$, $\sigma_2=1$, $\rho=0.6$, $\beta =0.5$, $a=0.3$, $\overline c_1=1.5$, $\overline c_2=1$ \\($u_1=0.33 < u_2=0.73 < w_0 = \infty$, corresponding to Theorem \ref{prop:u1<u2<w0})}
    \label{fig:row3}
    \vspace{0.1cm}
    \end{figure}

    \begin{figure}[H]
    \begin{subfigure}[H]{0.45\textwidth}
        \centering
        \includegraphics[width=\linewidth, trim = 0cm 0.5cm 1cm 2cm, clip = true]{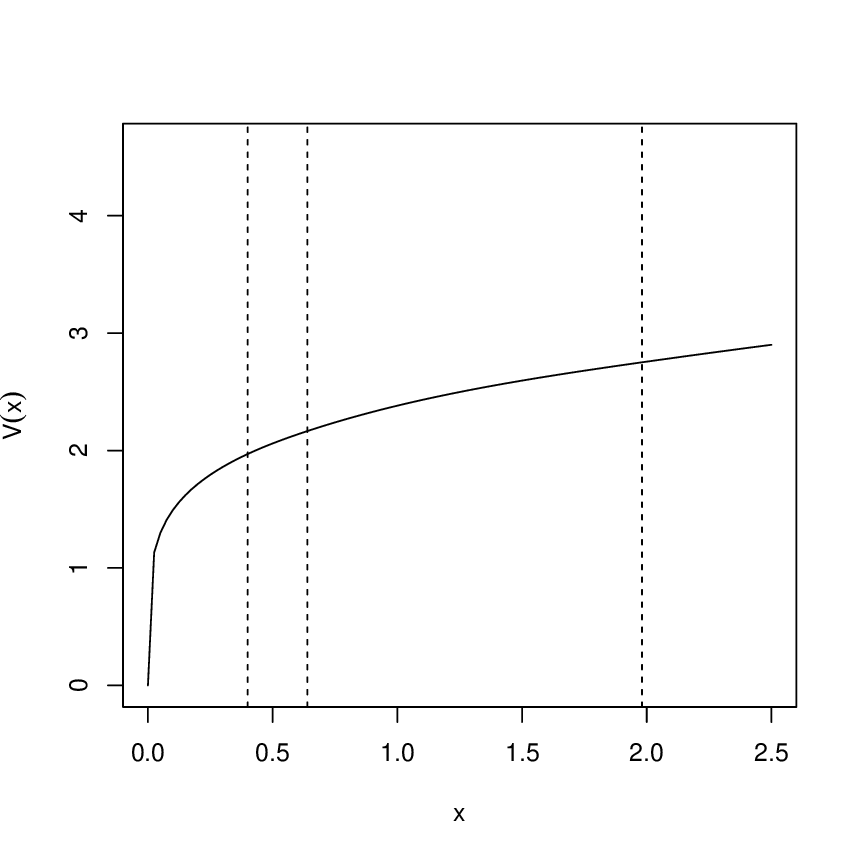}
        \label{fig:plot7}
    \end{subfigure}
    \hfill
    \begin{subfigure}[H]{0.45\textwidth}
        \centering
        \includegraphics[width=\linewidth, trim = 0cm 0.5cm 1cm 2cm, clip = true]{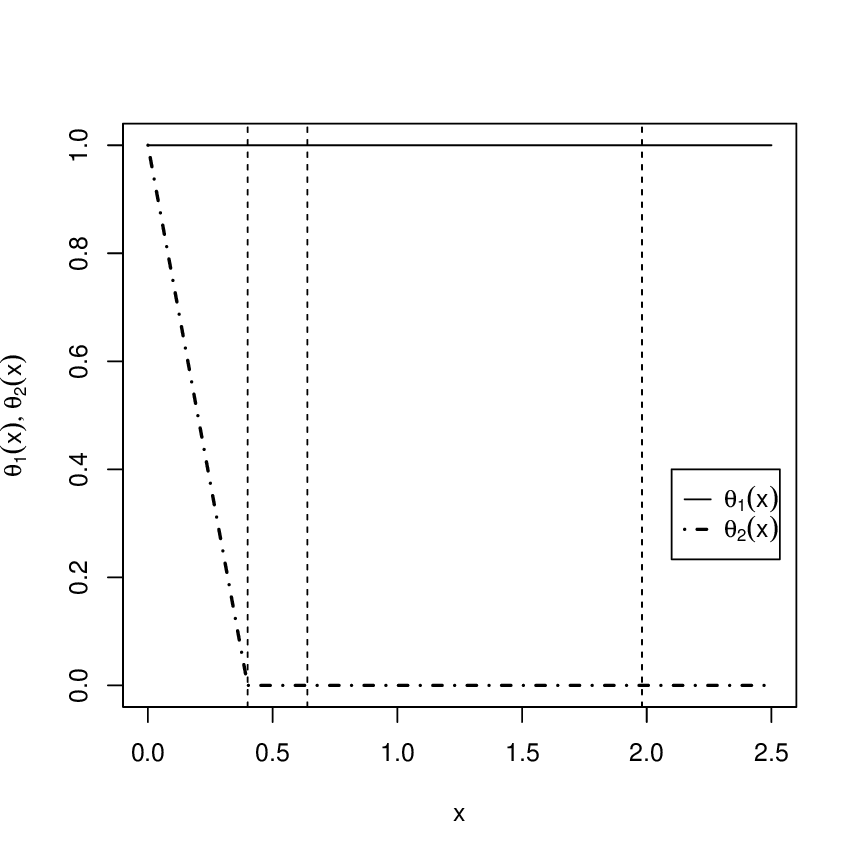}
        \label{fig:plot8}
    \end{subfigure}
    \\[-4ex]
   \caption{$\mu_1=1.5$, $\mu_2=2$, $\sigma_1=1.5$, $\sigma_2=1$, $\rho=0.6$, $\beta =0.5$, $a=0.3$, $\overline c_1=3$, $\overline c_2=2$ \\(The case of $0 < \frac{\mu_1 \slash \mu_2}{\sigma_1 \slash \sigma_2}\le\rho$, with $w_0=0.40$, $u_1=0.64$, and $u_2=1.98$)}
    \label{fig:row4}
     \vspace{-0.1cm}
    \end{figure}

    \begin{figure}[H]
    \begin{subfigure}[H]{0.45\textwidth}
        \centering
        \includegraphics[width=\linewidth, trim = 0cm 0.5cm 1cm 2cm, clip = true]{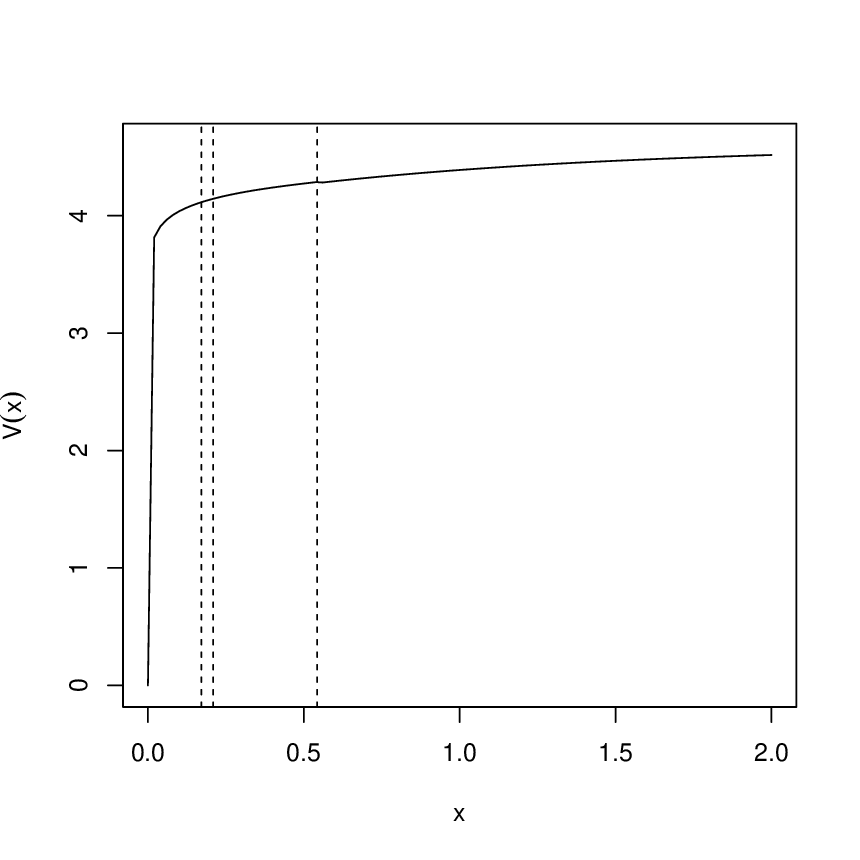}
        \label{fig:plot9}
    \end{subfigure}
    \hfill
    \begin{subfigure}[H]{0.45\textwidth}
        \centering
        \includegraphics[width=\linewidth, trim = 0cm 0.5cm 1cm 2cm, clip = true]{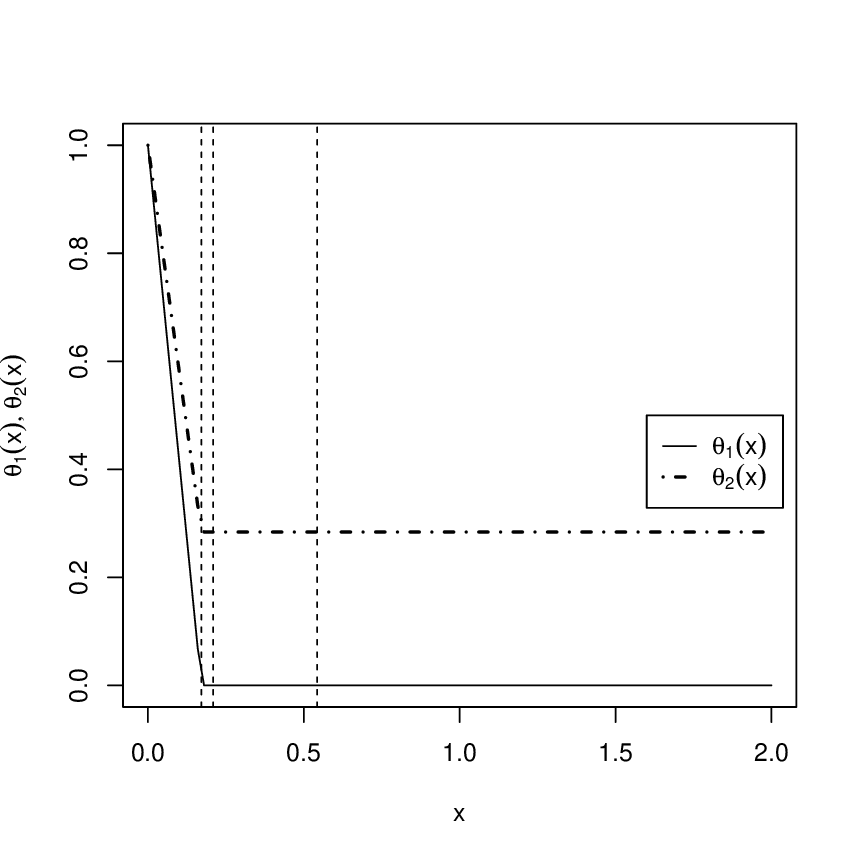}
        \label{fig:plot10}
    \end{subfigure}
    \\[-4ex]
   \caption{$\mu_1=2$, $\mu_2=4$, $\sigma_1=1$, $\sigma_2=1.5$, $\rho=-0.6$, $\beta =0.5$, $a=0.3$, $\overline c_1=3$, $\overline c_2=2$ \\(The case of $-1 < \rho \le 0$, with $w_0=0.17$, $u_1=0.21$, and $u_2=0.54$)}
    \label{fig:row5}
    \end{figure}
\FloatBarrier
    
    \section{Proofs of the Main Results}\label{section:proof of main}
    In this section, we provide the proofs for Theorems \ref{prop:w0<u1<u2}, \ref{prop:u1<w0<u2}, and \ref{prop:u1<u2<w0}.
    
	\subsection{Proof of Theorem \ref{prop:w0<u1<u2}}
    In this section, we present the key results used to obtain Theorem \ref{prop:w0<u1<u2}. The discussion serves as the proof for Theorem \ref{prop:w0<u1<u2}. 
    
    Let Assumption \ref{assume:w1<w2} hold. Moreover, suppose that $\overline c_1+\overline c_2\geq \frac{N_3N_2}{2N_1}$ and $\psi(\alpha_{0})\leq 0$, where $\overline c_i$ is the maximum dividend rate of Line $i$, as required in Theorem \ref{prop:w0<u1<u2}. 
   Recall that     $N_1,N_2,N_3$ are defined in \eqref{eqn:N1N2N3N4}, $\psi$ in \eqref{eqn:psi}, and $\alpha_0$ in \eqref{eqn:alpha_0}. 
    
    \subsubsection*{Deriving the analytical solution.}
    Suppose for now that $w_0\leq u_1\leq u_2$. In the region $\{x<w_0\}$, we must have $ C_i^*=0$ from \eqref{eqn:optimal dividend} for $i=1,2$, and the HJB equation \eqref{eqn:hjborig} becomes
    \begin{equation}\label{eqn:region 1}
			\mathcal{H}(x)-\beta g(x)=0,
    \end{equation}
    where $\mathcal{H}$ is defined in \eqref{eqn:optim reinsurance}. Substituting the candidate reinsurance strategies in \eqref{eqn:optimal_theta} to \eqref{eqn:region 1} yields
    \begin{equation}
        -M\frac{[g'(x)]^2}{g''(x)}-\beta g(x)=0,
    \end{equation}
    whose solution, denoted by $g_1$, is given by
    \begin{equation}
        g_1(x)=K_1x^{\gamma_1},
    \end{equation}
    where $M:=\frac{N_1\beta}{N_2-N_1}$, $\gamma_1$ is defined in \eqref{eqn:gamma1}, and $K_1>0$ is a constant. From \eqref{eqn:defn of w0}, we have $w_0=w_1$. 

    In the region $\{w_0<x<u_1\}$, we still have $(C_1^*, C^*_2)=(0,0)$, and the HJB equation becomes \eqref{eqn:region 1}. By the definition of $w_0$ and the constraint $\theta_1\in[0,1]$, it must hold that $\theta^*_1(x)=0$ for $x>w_0$. The following lemma solves the maximization problem \eqref{eqn:optim reinsurance} at the boundary of $\theta_1$.	\begin{lemma}\label{lemma:plateau}
		For $x>w_0$, 
		\begin{equation*}
	       (\theta^*_1,\theta^*_2)(x)=\left(0,1-\frac{w_0}{w_2}\right).
		\end{equation*}
    \end{lemma}
    \begin{proof}
     See Appendix \ref{app:plateau}.
    \end{proof}
	
        \begin{remark}
            Lemma \ref{lemma:plateau} suggests that when the total reserves exceed $w_0$, the reinsurance levels flattens and remains constant for both lines.    
        \end{remark} Using Lemma \ref{lemma:plateau}, \eqref{eqn:region 1} becomes 
	\begin{equation*}
		\frac{1}{2}N_4g''(x)+N_3g'(x)-\beta g(x)=0,
	\end{equation*}
	where $N_3$ and $N_4$ are defined in \eqref{eqn:N1N2N3N4}.
	The solution, denoted by $g_2$, is given by 
	\begin{equation*}
		g_2(x)=K_{2+}e^{\gamma_{2+} x}+K_{2-}e^{\gamma_{2-} x},
	\end{equation*}
	where $\gamma_{2\pm}$ are defined in \eqref{eqn:gamma234} and $K_{2\pm}$ are constants.
	
	In the region $\{u_1< x< u_2\}$, we must have $(\theta^*_1,\theta^*_2,C^*_1,C^*_2)=\left(0,1-\frac{w_0}{w_2},0,\overline c_2\right)$ using Lemma \ref{lemma:plateau}. The corresponding HJB equation becomes
	\begin{equation}\label{eqn:eqn3hjb}
		\frac{1}{2}N_4g''(x)+(N_3-\overline c_2)g'(x)-\beta g(x)+(1-a)\overline{c}_2=0,
	\end{equation}
	whose solution, denoted by $g_3$, is given by 
	\begin{equation*}
		g_3(x)=K_{3+}e^{\gamma_{3+} x}+K_{3-}e^{\gamma_{3-} x}+\frac{(1-a)\overline{c}_2}{\beta},
	\end{equation*}
	where $\gamma_{3\pm}$ are defined in \eqref{eqn:gamma234} and $K_{3\pm}$ are constants.
    
	In the region $\{x>u_2\}$, we must have, via Lemma \ref{lemma:plateau}, $(\theta^*_1,\theta^*_2,C^*_1,C^*_2)=\left(0,1-\frac{w_0}{w_2},\overline c_1,\overline c_2\right)$. The corresponding HJB equation becomes
	\begin{equation}\label{eqn:hjb4}
		\frac{1}{2}N_4g''(x)+(N_3-\overline c_1-\overline c_2)g'(x)-\beta g(x)+a\overline c_1+(1-a)\overline{c}_2=0,
	\end{equation}
	whose solution, denoted by $g_4$, which must satisfy $\lim_{x\to\infty}g(x)=\frac{a\overline{c}_1+(1-a)\overline{c}_2}{\beta}$, is given by 
	\begin{equation*}
		g_4(x)=K_{4-}e^{\gamma_{4-} x}+\frac{a\overline{c}_1+(1-a)\overline{c}_2}{\beta},
	\end{equation*}
	where $\gamma_{4-}$ is defined in \eqref{eqn:gamma234} and $K_{4-}$ is a constant. We then conjecture the following solution:
	\begin{equation*}
		g(x)=
		\begin{cases}
			K_1x^{\gamma_1}&\mbox{if $x<w_0$,}\\
			K_{2+}e^{\gamma_{2+} x}+K_{2-}e^{\gamma_{2-} x}&\mbox{if $w_0<x<u_1$,}\\
			K_{3+}e^{\gamma_{3+} x}+K_{3-}e^{\gamma_{3-} x}+\frac{(1-a)\overline{c}_2}{\beta}&\mbox{if $u_1<x<u_2$,}\\
			K_{4-}e^{\gamma_{4-} x}+\frac{a\overline{c}_1+(1-a)\overline{c}_2}{\beta}&\mbox{if $x>u_2$,}
		\end{cases}
	\end{equation*}
	where $K_1,K_{2\pm},K_{3\pm},K_{4-},u_1,u_2$ are unknown constants. To ensure that $g$ is twice continuously differentiable, we require $g$, $g'$, and $g''$ to be continuous at the switching points $w_0,u_1,u_2$.
	
	Since in the neighborhood of $w_0$ the function $g$ satisfies \eqref{eqn:region 1} with $(\theta^*_1,\theta^*_2)=(0,0)$, it suffices to show that $g$ and $g'$ are continuous at $w_0$. That is, if we can show that $g_1(w_0)=g_2(w_0)$ and $g'_1(w_0)=g'_2(w_0)$, then by \eqref{eqn:region 1}, it immediately follows that $g''_1(w_0)=g''_2(w_0)$. Let $\alpha_{2+}:=\frac{K_{2+}}{K_1}e^{\gamma_{2+}w_0}$ and $\alpha_{2-}:=\frac{K_{2-}}{K_1}e^{\gamma_{2-}w_0}$. We then have the following system of equations:
	\begin{equation}\label{eqn:system at w_0}
		\begin{aligned}
			\alpha_{2+}+\alpha_{2-}&=w_0^{\gamma_1}\\
			\gamma_{2+}\alpha_{2+}+\gamma_{2-}\alpha_{2-}&=\gamma_1w_0^{\gamma_1-1},
		\end{aligned}
	\end{equation}
	whose solution is given by
	\begin{equation}
		(\alpha_{2+},\alpha_{2-})=\left(\frac{w_0^{\gamma_1-1}(\gamma_1-\gamma_{2-}w_0)}{\gamma_{2+}-\gamma_{2-}},\frac{w_0^{\gamma_1-1}(\gamma_{2+}w_0-\gamma_1)}{\gamma_{2+}-\gamma_{2-}}\right).
	\end{equation}
	Using \eqref{eqn:N1N2N3N4}, we have
	\begin{equation*}
		w_0=\frac{N_3}{2\beta}\left(1-\frac{N_1}{N_2}\right)\quad\mbox{and}\quad N_4=\frac{N_3^2}{2\beta}\left(\frac{N_2}{N_1}-1\right).
	\end{equation*}
	Since $N_1<N_2$, we have the following:
	\begin{equation}\label{eqn:simplify alpha 2s}
		\begin{aligned}
			\gamma_{2+}-\gamma_{2-}&=\frac{4\beta\sqrt{N_1N_2}}{N_3(N_2-N_1)}>0\\
			\gamma_1-\gamma_{2-}w_0&=1+\sqrt{\frac{N_1}{N_2}}>0\\
			\gamma_{2+}w_0-\gamma_1&=\sqrt{\frac{N_1}{N_2}}-1<0.
		\end{aligned}
	\end{equation}
	Using $\gamma_{2\pm}=\frac{2\beta N_1}{N_3(N_2-N_1)}\left(-1\pm\sqrt{\frac{N_2}{N_1}}\right)$, we can rewrite $\alpha_{2\pm}$ as 
	\begin{equation}\label{eqn:alpha2+-}
		\alpha_{2+}=-\frac{w_0^{\gamma_1}N_3(N_2-N_1)}{4\beta N_1}\gamma_{2-}>0\quad\mbox{and}\quad\alpha_{2-}=-\frac{w_0^{\gamma_1}N_3(N_2-N_1)}{4\beta N_1}\gamma_{2+}<0.
	\end{equation}
	Consequently,
	\begin{equation*}
		g_2(x)=-\lambda\left[\gamma_{2-}e^{\gamma_{2+}(x-w_0)}+\gamma_{2+}e^{\gamma_{2-}(x-w_0)}\right],
	\end{equation*}
	where $\lambda=\frac{K_1w_0^{\gamma_1}N_3(N_2-N_1)}{4\beta N_1}$ and $K_1>0$ is still unknown.
	
	Let $\alpha_{3+}:=K_{3+}e^{\gamma_{3+}u_1}$ and $\alpha_{3-}:=K_{3-}e^{\gamma_{3-}u_1}$. We obtain $g_3$ in an alternative expression by
	\begin{equation*}
		g_3(x)=\alpha_{3+}e^{\gamma_{3+}(x-u_1)}+\alpha_{3-}e^{\gamma_{3-}(x-u_1)}+\frac{(1-a)\overline{c}_2}{\beta}.
	\end{equation*}By the definition of $u_1$, we have $g'(u_1)=1-a$. Hence,
	\begin{equation*}
		1-a=g'_3(u_1)=\gamma_{3+}\alpha_{3+}+\gamma_{3-}\alpha_{3-},
	\end{equation*}
	or, equivalently,
	\begin{equation}\label{eqn:g'(u1)=1-a}
		\alpha_{3+}=\frac{1}{\gamma_{3+}}\left(1-a-\gamma_{3-}\alpha_{3-}\right).
	\end{equation}
	Similar to the first switching point $w_0$, it suffices to show that $g_2(u_1)=g_3(u_1)$ and $g'_2(u_1)=g'_3(u_1)$ to ensure that $g$ is twice continuously differentiable at $x=u_1$. This requirement yields the following system of equations:
	\begin{equation}\label{eqn:system@u1}
		\begin{aligned}
			\lambda\left(-\gamma_{2-}e^{\gamma_{2+}(u_1-w_0)}-\gamma_{2+}e^{\gamma_{2-}(u_1-w_0)}\right)&=(1-a)\left(\frac{1}{\gamma_{3+}}+\frac{\overline c_2}{\beta}\right)+\left(1-\frac{\gamma_{3-}}{\gamma_{3+}}\right)\alpha_{3-}\\
			\lambda\left(-\gamma_{2+}\gamma_{2-}e^{\gamma_{2+}(u_1-w_0)}-\gamma_{2+}\gamma_{2-}e^{\gamma_{2-}(u_1-w_0)}\right)&=1-a,
		\end{aligned}
	\end{equation}
	where we have used \eqref{eqn:g'(u1)=1-a} to derive the first equation. From the second equation in \eqref{eqn:system@u1}, we obtain $\lambda$ given by \eqref{eqn:lambda}. Dividing the first equation in \eqref{eqn:system@u1} by the second equation yields
	\begin{equation*}
		\frac{\gamma_{2-}e^{\gamma_{2+}(u_1-w_0)}+\gamma_{2+}e^{\gamma_{2-}(u_1-w_0)}}{\gamma_{2+}\gamma_{2-}e^{\gamma_{2+}(u_1-w_0)}+\gamma_{2+}\gamma_{2-}e^{\gamma_{2-}(u_1-w_0)}}=\alpha_3,
	\end{equation*}
	which can be rewritten as
    \begin{equation}\label{eqn:u1}
		u_1=w_0+\frac{1}{\gamma_{2+}-\gamma_{2-}}\ln\left(\frac{\alpha_{2-}(\gamma_{2-}\alpha_3-1)}{\alpha_{2+}(1-\gamma_{2+}\alpha_3)}\right),
	\end{equation} where $\alpha_3$ is defined in \eqref{eqn:alpha_3}. 
    We point out that the well-definedness of $u_1$ has not been established yet. 
    
    Let $\alpha_{4-}:=K_{4-}e^{\gamma_{4-}u_2}$. By the definition of $u_2$, we have $g'(u_2)=a$. Hence,
	\begin{equation*}
		a=g'_4(u_2)=\gamma_{4-}\alpha_{4-},
	\end{equation*}
	or, equivalently,
	\begin{equation}\label{eqn:alpha4-}
		\alpha_{4-}=\frac{a}{\gamma_{4-}}.
	\end{equation}
	We can then rewrite $g_4$ as
	\begin{equation*}
		g_4(x)=\frac{a}{\gamma_{4-}}e^{\gamma_{4-}(x-u_2)}+\frac{a\overline c_1+(1-a)\overline{c}_2}{\beta}.
	\end{equation*}
Similar to the arguments above, it suffices to show that $g'_3(u_2)=g'_4(u_2)$ and $g''_3(u_2)=g''_4(u_2)$, to ensure that $g$ is twice continuously differentiable at $x=u_2$. At $x=u_2$, we have the following system of equations:
	\begin{equation}\label{eqn:system at u_2}
		\begin{aligned}
			\gamma_{3+}\alpha_{3+}e^{\gamma_{3+}(u_2-u_1)}+\gamma_{3-}\alpha_{3-}e^{\gamma_{3-}(u_2-u_1)}&=a\\
			\gamma^2_{3+}\alpha_{3+}e^{\gamma_{3+}(u_2-u_1)}+\gamma^2_{3-}\alpha_{3-}e^{\gamma_{3-}(u_2-u_1)}&=\gamma_{4-}a.
		\end{aligned}
	\end{equation}
	Dividing the first equation by the second equation yields
	\begin{equation*}
		\frac{\gamma^2_{3+}\alpha_{3+}e^{\gamma_{3+}(u_2-u_1)}+\gamma^2_{3-}\alpha_{3-}e^{\gamma_{3-}(u_2-u_1)}}{\gamma_{3+}\alpha_{3+}e^{\gamma_{3+}(u_2-u_1)}+\gamma_{3-}\alpha_{3-}e^{\gamma_{3-}(u_2-u_1)}}=\gamma_{4-},
	\end{equation*}
	which is equivalent to
	\begin{equation}\label{eqn:u2}
		u_2=u_1+\frac{1}{\gamma_{3+}-\gamma_{3-}}\ln\left(\frac{\alpha_{3-}\gamma_{3-}(\gamma_{4-}-\gamma_{3-})}{\alpha_{3+}\gamma_{3+}(\gamma_{3+}-\gamma_{4-})}\right).
	\end{equation}
    Thus, we have obtained the form of the candidate value function $g$ defined in \eqref{eqn:g first} and the formulas for $u_1$ and $u_2$ in \eqref{eqn:u1} and \eqref{eqn:u2}, respectively. However, it must be noted that we have not yet (i) solved for $\alpha_{3-}$, which implies that $u_1$ and $u_2$ may not be well-defined, (ii) guaranteed that $w_0\leq u_1\leq u_2$, and (iii) shown that $g$ is increasing and concave. To resolve these issues, the next steps are to establish the bounds for $\alpha_{3-}$ and to solve for $\alpha_{3-}$ using these bounds.

    \subsubsection*{Establishing the bounds for $\alpha_{3-}$.}

    Since the candidate value function $g$ must be positive for $x>0$, we must have $\alpha_3>0$ from the first equation in \eqref{eqn:system@u1}. Next, it follows from $\alpha_{2-}<0<\alpha_{2+}$ and \eqref{eqn:u1} that $1-\gamma_{2+}\alpha_3>0$. Combining these inequalities for $\alpha_3$ yields $0<\alpha_3<\frac{1}{\gamma_{2+}}$, which is equivalent to
    \begin{equation}\label{eqn:alpha 1st ineq}
		\underline{\alpha}:=-\frac{(1-a)\gamma_{3+}}{\gamma_{3+}-\gamma_{3-}}\left(\frac{1}{\gamma_{3+}}+\frac{\overline{c}_2}{\beta}\right)<\alpha_{3-}<\frac{(1-a)\gamma_{3+}}{\gamma_{3+}-\gamma_{3-}}\left(\frac{1}{\gamma_{2+}}-\frac{1}{\gamma_{3+}}-\frac{\overline{c}_2}{\beta}\right)=:\overline{\alpha}.
	\end{equation}
    It is clear that $\underline\alpha<\overline\alpha$ since $\frac{1}{\gamma_{2+}}>0$. 
    
    We have now established that \eqref{eqn:u1} is well-defined, but it does not guarantee that $w_0\leq u_1$. The following result gives a necessary and sufficient condition to ensure that $w_0\leq u_1$.
	\begin{lemma}\label{lemma:w_0<u_1} 
		$w_0\leq u_1$ if and only if
		\begin{equation*}
			\alpha_{3-}\geq \alpha_0,
		\end{equation*}
        where $\alpha_0$ is defined in \eqref{eqn:alpha_0}.
	\end{lemma}
	\begin{proof}
		From \eqref{eqn:u1}, $w_0\leq u_1$ is equivalent to
		\begin{equation*}
			\frac{\alpha_{2-}(\gamma_{2-}\alpha_3-1)}{\alpha_{2+}(1-\gamma_{2+}\alpha_3)}\geq 1.
		\end{equation*}
		Define $w(x):=\frac{\alpha_{2-}(\gamma_{2-}x-1)}{\alpha_{2+}(1-\gamma_{2+}x)}$. We compute
		\begin{equation*}
			w'(x)=\frac{\alpha_{2+}\alpha_{2-}\gamma_{2-}(1-\gamma_{2+}\alpha_3)+\alpha_{2+}\alpha_{2-}\gamma_{2+}(\gamma_{2-}\alpha_3-1)}{\alpha_{2+}^2(1-\gamma_{2+}\alpha_3)^2}>0,
		\end{equation*}
		which implies that $w$ is a strictly increasing function. Moreover, $w$ is continuous on $\left(0,\frac{1}{\gamma_{2+}}\right)$, $\lim_{x\downarrow 0}w(x)=-\frac{\alpha_{2-}}{\alpha_{2+}}>0$, and $\lim_{x\uparrow \frac{1}{\gamma_{2+}}}w(x)=\infty$. From \eqref{eqn:system at w_0}, $\alpha_{2+}+\alpha_{2-}=w_0^{\gamma_1}>0$, which is equivalent to $-\frac{\alpha_{2-}}{\alpha_{2+}}<1$. By the intermediate value theorem and the strict monotonicity of $w$, there exists a unique $x_0\in\left(0,\frac{1}{\gamma_{2+}}\right)$ such that $w(x_0)=1$. Using \eqref{eqn:system at w_0} and $w(x_0)=1$, we obtain
		\begin{equation*}
			x_0=\frac{\alpha_{2+}+\alpha_{2-}}{\alpha_{2-}\gamma_{2-}+\alpha_{2+}\gamma_{2+}}=\frac{w_0^{\gamma_1}}{\gamma_1w_0^{\gamma_1-1}}=\frac{w_0}{\gamma_1}=\frac{N_3}{2\beta}.
		\end{equation*}
		This implies that $w_0\leq u_1$ and $\alpha_3\geq \frac{N_3}{2\beta}$ are equivalent, which completes the proof.
	\end{proof}
    
    Since $\frac{N_3}{2\beta}>0$, we have $\underline\alpha <\alpha_0$. From \eqref{eqn:N1N2N3N4} and the third inequality in \eqref{eqn:simplify alpha 2s}, we get $\frac{N_3}{2\beta}=\frac{w_0}{\gamma_1}<\frac{1}{\gamma_{2+}}$. Hence, $\alpha_{3-}$ must now satisfy the following inequalities:
    \begin{equation}\label{eqn:alpha 2nd inequality}
        \alpha_0\leq \alpha_{3-} <\overline \alpha.
    \end{equation}

    Our next agenda is to determine the signs of $\alpha_{3\pm}$, both of which have yet to be solved. First, we note that a sufficient condition such that $g$ is increasing, particularly in the region $\{u_1<x<u_2\}$, is $\alpha_{3-}\le 0\le\alpha_{3+}$. The following lemma proves that $\alpha_{3-}<0$.
	\begin{lemma}\label{lemma:alpha3-<0}
		$\alpha_{3-}<0$.
	\end{lemma}
	\begin{proof}
		From \eqref{eqn:alpha 1st ineq}, it suffices to show that
		\begin{equation*}
			\frac{\overline{c}_2}{\beta}\geq  \frac{1}{\gamma_{2+}}-\frac{1}{\gamma_{3+}}.
		\end{equation*}
		Write $k:=2\beta N_4>0$. We use proof by contradiction. Suppose $\frac{\overline{c}_2}{\beta}< \frac{1}{\gamma_{2+}}-\frac{1}{\gamma_{3+}}$. This is equivalent to
		\begin{equation}\label{eqn:r1}
			\begin{aligned}
				&2\overline c_2\sqrt{(N_3-\overline c_2)^2+k}\cdot\sqrt{N_3^2+k}-\overline c_2\left(k-2N_3(N_3-\overline{c}_2)\right)\\
				&< (2N_3\overline c_2+k)\sqrt{(N_3-\overline c_2)^2+k}+(2\overline c_2(N_3-\overline c_2)-k)\sqrt{N_3^2+k}.
			\end{aligned}			
		\end{equation}
		It can be shown that the left-hand side of \eqref{eqn:r1} is always positive. If the right-hand side of \eqref{eqn:r1} is negative, then it leads to a contradiction. Otherwise, we can square both sides and obtain
		\begin{equation*}
			\overline{c}_2^2k< 0,
		\end{equation*}
		which is also a contradiction. This completes the proof.
	\end{proof}
    We must now ensure that $\alpha_{3+}\geq 0$. Using \eqref{eqn:g'(u1)=1-a}, we must have the following:
	\begin{equation*}
		\alpha_{3-}\geq \frac{1-a}{\gamma_{3-}}.
	\end{equation*}
    We now check if the above inequality updates the lower bound in \eqref{eqn:alpha 2nd inequality}; that is, we determine whether $\frac{1-a}{\gamma_{3-}}$ is a larger lower bound than $\alpha_0$. The following lemma gives a necessary and sufficient condition such that $\frac{1-a}{\gamma_{3-}}$ tightens the bounds for $\alpha_{3-}$.
    \begin{lemma}\label{lemma: LB for alpha3}
		$\alpha_0>\frac{1-a}{\gamma_{3-}}$ if and only if $\overline{c}_2<\frac{N_3N_2}{2N_1}=\frac{\beta w_1}{\gamma_1(1-\gamma_1)}$.
	\end{lemma}
	\begin{proof}
		Suppose $\alpha_0>\frac{1-a}{\gamma_{3-}}$. This inequality is equivalent to
		\begin{equation}\label{eqn:rineq1}
			(N_3-2\overline c_2)(N_3-\overline{c}_2)+N_3^2\left(\frac{N_2}{N_1}-1\right)>-(N_3-2\overline c_2)\sqrt{(N_3-\overline c_2)^2+N_3^2\left(\frac{N_2}{N_1}-1\right)}.
		\end{equation}
		It can easily be seen that the above inequality is always satisfied if $\overline{c}_2<\frac{N_3}{2}$. Suppose $\overline c_2\geq \frac{N_3}{2}$. Squaring both sides of \eqref{eqn:rineq1} yields
		\begin{equation*}
			2(N_3-2\overline c_2)(N_3-\overline{c}_2)+N_3^2\left(\frac{N_2}{N_1}-1\right)>(N_3-2\overline c_2)^2,
		\end{equation*}
		which is equivalent to
		\begin{equation*}
			\overline c_2<\frac{N_3N_2}{2N_1}.
		\end{equation*}
		However, we need to ensure that the left-hand side of \eqref{eqn:rineq1} is positive for $\frac{N_3}{2}\leq \overline c_2<\frac{N_3N_2}{2N_1}$. Define $h(x):=(N_3-2x)(N_3-x)+N_3^2\left(\frac{N_2}{N_1}-1\right)$. It follows that $h$ achieves its minimum at $x=\frac{3N_3}{4}$. Moreover, $h\left(\frac{N_3}{2}\right)=N_3^2\left(\frac{N_2}{N_1}-1\right)>0$, $h\left(\frac{N_3N_2}{2N_1}\right)=\frac{N_3^2N_2}{2N_1}\left(\frac{N_2}{N_1}-1\right)>0$, and $h\left(\frac{3N_3}{4}\right)=N_3^2\left(\frac{N_2}{N_1}-\frac{9}{8}\right)$. We also have $\frac{3N_3}{4}\geq \frac{N_3N_2}{2N_1}$ if and only if $\frac{N_2}{N_1}\leq\frac{3}{2}$. If $\frac{N_2}{N_1}>\frac{3}{2}>\frac{9}{8}$, then $h\left(\frac{3N_3}{4}\right)>0$, which implies that $h(x)>0$ for any $x$. If $\frac{N_2}{N_1}\leq\frac{3}{2}$, then $h(x)>0$ on $\left[\frac{N_3}{2},\frac{N_3N_2}{2N_1}\right]$. Thus, the left side of the inequality in \eqref{eqn:rineq1} is always positive for $\overline c_2\in\left[\frac{N_3}{2},\frac{N_3N_2}{2N_1}\right]$, which makes squaring both sides of \eqref{eqn:rineq1} valid. From \eqref{eqn:N1N2N3N4}, we have $\frac{N_3N_2}{2N_1}=\frac{\beta w_1}{\gamma_1(1-\gamma_1)}$. This completes the proof.
	\end{proof}
        The following lemma proves that $\frac{1-a}{\gamma_{3-}}$ does not exceed the upper bound $\overline\alpha$.    
        \begin{lemma}\label{lemma for alpha3 v2}
		$\frac{1-a}{\gamma_{3-}}<\overline{\alpha}$.
	\end{lemma}
	\begin{proof}
		If $\overline c_2< \frac{N_3N_2}{2N_1}$, the result follows immediately from Lemma \ref{lemma: LB for alpha3} and the fact that $\alpha_0<\overline{\alpha}$. Next, we consider the case of $\overline c_2\geq  \frac{N_3N_2}{2N_1}$, under which $\frac{1-a}{\gamma_{3-}}>\alpha_0$ by Lemma \ref{lemma: LB for alpha3}. 
       We want to prove that $\frac{1-a}{\gamma_{3-}}<\overline{\alpha}$. It can be shown that $\frac{1-a}{\gamma_{3-}}<\overline{\alpha}$ is equivalent to 
		\begin{equation}\label{eqn:c2/b<UB}
			\frac{\overline{c}_2}{\beta}<  \frac{1}{\gamma_{2+}}-\frac{1}{\gamma_{3-}}.
		\end{equation}
		We will prove \eqref{eqn:c2/b<UB} instead. From the elementary inequality
		\begin{equation*}
			\sqrt{A^2+B}-A<\frac{B}{2A},\quad A,B>0,
		\end{equation*}
		we obtain the following:
		\begin{equation*}
			\gamma_{2+}<\frac{\beta}{N_3}\Leftrightarrow\frac{1}{\gamma_{2+}}>\frac{N_3}{\beta}.
		\end{equation*}
		If $N_3-\overline{c}_2\leq 0$, then
		\begin{equation*}
			-\gamma_{3-}<\frac{\beta}{\overline{c}_2-N_3}\Leftrightarrow-\frac{1}{\gamma_{3-}}>\frac{\overline c_2-N_3}{\beta},
		\end{equation*}
		and \eqref{eqn:c2/b<UB} follows. If $N_3-\overline{c}_2> 0$, then
		\begin{equation*}
			\gamma_{2+}<\frac{\beta}{N_3}<\frac{\beta}{\overline{c}_2}\Leftrightarrow\frac{\overline{c}_2}{\beta}<\frac{1}{\gamma_{2+}},
		\end{equation*}
		and \eqref{eqn:c2/b<UB} follows since $\gamma_{3-}<0$. The proof is complete.
	\end{proof}

	Lemmas \ref{lemma: LB for alpha3} and \ref{lemma for alpha3 v2} imply that $\alpha_{3-}$ must satisfy the following inequalities so that $\alpha_{3+}> 0$:
	\begin{equation}\label{eqn:alpha3- bounds}
		\alpha_{LB}< \alpha_{3-}<\overline{\alpha},
	\end{equation}
    where $\alpha_{LB}$ is defined in \eqref{eqn: alphaLB & UB}.
    \begin{remark}
        It also follows that $g$ is strictly increasing in $\{u_1<x<u_2\}$. It is easy to see that $g$ is strictly increasing in the other regions. Hence, due to the continuity of the first derivative, $g$ is strictly increasing for $x>0$.

        Moreover, it is easy to show that $g$ is strictly concave in the regions $\{x<w_0\}$ and $\{x > u_2\}$. Since $g'''(x)>0$ for $\{w_0<x<u_1\}$ and $\{u_1<x<u_2\}$, the strict concavity of $g$ on $[w_0,u_2]$ follows from the continuity of the second derivative.
    \end{remark}
    We have already established that $\alpha_{3-}<0<\alpha_{3+}$. Combining this with the fact that $\gamma_{3-}<0<\gamma_{3+}$ and $\gamma_{3-}<\gamma_{4-}<0$ implies that $u_2$ given in \eqref{eqn:u2} is well-defined. However, it does not guarantee that $u_1\leq u_2$. The following lemma gives a necessary and sufficient condition for $u_1\leq u_2$ to hold.
	
	\begin{lemma}\label{lemma:u_1<u_2}
		$u_1\leq u_2$ if and only if
		\begin{equation*}
			\alpha_{3-}\leq \alpha_{UB},
		\end{equation*}
            where $\alpha_{UB}$ is defined in \eqref{eqn: alphaLB & UB}.
	\end{lemma}
	\begin{proof}
		$u_1\leq u_2$ is equivalent to $\frac{\alpha_{3-}\gamma_{3-}(\gamma_{4-}-\gamma_{3-})}{\alpha_{3+}\gamma_{3+}(\gamma_{3+}-\gamma_{4-})}\geq 1$. Using \eqref{eqn:g'(u1)=1-a} yields the desired inequality.
	\end{proof}
    The following lemma proves that $\alpha_{UB}$ tightens the bounds of $\alpha_{3-}$ in \eqref{eqn:alpha3- bounds}.
	\begin{lemma}\label{lemma:UB}
		$\alpha_{UB}<\overline{\alpha}$.
	\end{lemma}
	\begin{proof}
		It is sufficient to prove the following:
		\begin{equation*}
			-(\gamma_{3+}+\gamma_{3-})+\gamma_{4-}-\gamma_{3+}\gamma_{3-}\frac{\overline c_2}{\beta}+\frac{\gamma_{3+}\gamma_{3-}}{\gamma_{2+}}<0,
		\end{equation*}
		which is equivalent to
		\begin{equation*}
			\begin{aligned}
				&N_3\sqrt{(N_3-\overline c_1-\overline c_2)^2+2\beta N_4}+(N_3+\overline c_1+\overline c_2)\sqrt{N_3^2+2\beta N_4}\\
				&<2\beta N_4+N_3(N_3+\overline{c}_1+\overline{c}_2)+\sqrt{N_3^2+2\beta N_4}\sqrt{(N_3-\overline c_1-\overline c_2)^2+2\beta N_4}.
			\end{aligned}
		\end{equation*}
		Squaring both sides and combining terms yields
		\begin{equation*}
			\sqrt{N_3^2+2\beta N_4}\sqrt{(N_3-\overline c_1-\overline c_2)^2+2\beta N_4}> -N_3(N_3-\overline c_1-\overline c_2)-2\beta N_4=N_3\left(\overline c_1+\overline c_2-\frac{N_3N_2}{N_1}\right),
		\end{equation*}
		which is always true if $\overline c_1+\overline c_2<\frac{N_3N_2}{N_1}$. Suppose $\overline c_1+\overline c_2\geq\frac{N_3N_2}{N_1}$. Squaring both sides of the above inequality yields $(\overline c_1+\overline c_2)^2>0$, which is also always true. The proof is complete.
	\end{proof}
    The following lemma gives a necessary and sufficient condition such that $\alpha_{LB}\leq \alpha_{UB}$.
	\begin{lemma}\label{lemma:LB<UB}
		$\alpha_{LB}\leq \alpha_{UB}$ if and only if
		\begin{equation*}
			\overline c_1+\overline c_2\geq \frac{N_3N_2}{2N_1}=\frac{\beta w_1}{\gamma_1(1-\gamma_1)}.
		\end{equation*}
	\end{lemma}
	\begin{proof}
		If $\overline c_2\geq \frac{N_3N_2}{2N_1}$, then, from the proof of Lemma \ref{lemma:u_1<u_2}, $\frac{1-a}{\gamma_{3-}}=\alpha_{LB}\leq \alpha_{UB}$. Moreover, $\overline c_2\geq \frac{N_3N_2}{2N_1}$ and $\overline c_1\geq 0$ imply that $\overline c_1+\overline c_2\geq\frac{N_3N_2}{2N_1}$. Suppose now that $\overline c_2<\frac{N_3N_2}{2N_1}$, which implies that $\alpha_{LB}=\alpha_0$. Then, $\alpha_{LB}\leq\alpha_{UB}$ is equivalent to
		\begin{equation*}
			-(\gamma_{3+}+\gamma_{3-})+\gamma_{4-}-\gamma_{3+}\gamma_{3-}\left(\frac{\overline{c}_2}{\beta}-\frac{N_3}{2\beta}\right)\geq 0,
		\end{equation*}
		which is also equivalent to
		\begin{equation*}
			\overline{c}_1+\overline{c}_2\geq\sqrt{(N_3-\overline c_1-\overline c_2)^2+2\beta N_4}.
		\end{equation*}
		Squaring both sides yields
		\begin{equation*}
			\overline{c}_1+\overline c_2 \geq \frac{N_3^2+2\beta N_4}{2N_3}=\frac{N_3N_2}{2N_1}=\frac{\beta w_1}{\gamma_1(1-\gamma_1)},
		\end{equation*}
		which proves the result.
	\end{proof}
    
    Using $\overline c_1+\overline c_2\geq \frac{N_3N_2}{2N_1}$ and Lemma \ref{lemma:LB<UB}, we have $\alpha_{LB}\leq \alpha_{UB}$. We now state the following result which gives a necessary and sufficient condition such that $w_0\leq u_1\leq u_2$.
	
	\begin{lemma}\label{lemma:w0<u1<u2}
		$w_0\leq u_1\leq u_2$ if and only if $\alpha_{3-}\in(\alpha_{LB},\alpha_{UB})$.
	\end{lemma}
	\begin{proof}
		This is a direct consequence of Lemmas \ref{lemma:w_0<u_1}, \ref{lemma:u_1<u_2} and \ref{lemma:LB<UB}.
	\end{proof}
    It remains to prove the existence of $\alpha_{3-}$.
    \subsubsection*{Solving for $\alpha_{3-}$.}
     Although we have combined the two equations in \eqref{eqn:system at u_2} to obtain \eqref{eqn:u2}, $\alpha_{3-}$ still remains to be determined. We do this now via the first equation in \eqref{eqn:system at u_2}, which ensures that the first derivative is continuous at $x=u_2$. This motivates the form of $\psi$ defined in \eqref{eqn:psi}. The following lemma gives a necessary and sufficient condition such that $\alpha_{3-}$ exists and is unique.
    \begin{lemma}\label{lemma:exist alpha3}
        $\alpha_{3-}$ is the unique solution to $\psi(z)=0$ on $(\alpha_{LB},\alpha_{UB})$ if and only if $\psi(\alpha_{LB})\leq 0$, where $\psi$ is defined in \eqref{eqn:psi}.
    \end{lemma}
    \begin{proof}
        We can rewrite $\psi$ as follows:
        \begin{equation*}
        \begin{aligned}
            \psi(z)
            &=(1-a-\gamma_{3-}z)\left[\frac{\gamma_{3-}(\gamma_{4-}-\gamma_{3-})z}{(1-a-\gamma_{3-}z)(\gamma_{3+}-\gamma_{4-})}\right]^{\frac{\gamma_{3+}}{\gamma_{3+}-\gamma_{3-}}}+\gamma_{3-}ze^{\gamma_{3-}\zeta(z)}-a\\
            &=(1-a-\gamma_{3-}z)^{-\frac{\gamma_{3-}}{\gamma_{3+}-\gamma_{3-}}}\left[\frac{\gamma_{3-}(\gamma_{4-}-\gamma_{3-})z}{\gamma_{3+}-\gamma_{4-}}\right]^{\frac{\gamma_{3+}}{\gamma_{3+}-\gamma_{3-}}}+\gamma_{3-}ze^{\gamma_{3-}\zeta(z)}-a.
        \end{aligned}
        \end{equation*}
        Since $\gamma_{3-}<0$ and $\lim_{z\downarrow \frac{1-a}{\gamma_{3-}}}\zeta(z)=+\infty$,  $\lim_{z\downarrow \frac{1-a}{\gamma_{3-}}}\psi(z)=-a<0$. From Lemma \ref{lemma:u_1<u_2}, $\zeta(\alpha_{UB})=0$. Because $a\leq \frac{1}{2}$, $\psi(\alpha_{UB})=1-2a>0$. Hence, by the intermediate value theorem, there exists a $z_0\in\left(\frac{1-a}{\gamma_{3-}},\alpha_{UB}\right)$ such that $\psi(z_0)=0$. We now prove the uniqueness of $z_0$. From the definitions of $\zeta$ and $\psi$ in \eqref{eqn:psi}, we have
        \begin{equation*}
        \begin{aligned}
            \zeta'(z)&=\frac{1-a}{(\gamma_{3+}-\gamma_{3-})(1-a-\gamma_{3-}z)z},\\
            \psi'(z)&=\gamma_{3-}\left(e^{\gamma_{3-}\zeta(z)}-e^{\gamma_{3+}\zeta(z)}\right)+\frac{1-a}{\gamma_{3+}-\gamma_{3-}}\left[\frac{\gamma_{3+}}{z}e^{\gamma_{3+}\zeta(z)}+\frac{\gamma_{3-}^2}{1-a-\gamma_{3-}z}e^{\gamma_{3-}\zeta(z)}\right].    
        \end{aligned}
        \end{equation*}
        By $\gamma_{3+}>\gamma_{3-}$,
        \begin{equation*}
            (\gamma_{3+}-\gamma_{3-})\zeta(z)=\ln\left(\frac{\gamma_{3-}(\gamma_{4-}-\gamma_{3-})z}{(1-a-\gamma_{3-}z)(\gamma_{3+}-\gamma_{4-})}\right)<\ln\left(\frac{-\gamma_{3-}^2z}{\gamma_{3+}(1-a-\gamma_{3-}z)}\right).
        \end{equation*}
        It follows that
        \begin{equation*}
            \frac{\gamma_{3+}}{z}e^{\gamma_{3+}\zeta(z)}>\frac{-\gamma_{3-}^2}{1-a-\gamma_{3-}z}e^{\gamma_{3-}\zeta(z)}.
        \end{equation*}
        Now using $\zeta(z)\geq 0$ for $z\in\left(\frac{1-a}{\gamma_{3-}},\alpha_{UB}\right)$, we obtain  $\psi'(z)>0$ on $\left(\frac{1-a}{\gamma_{3-}},\alpha_{UB}\right)$. The uniqueness of $z_0$ follows immediately.

        If $\overline c_2\geq \frac{N_3N_2}{2N_1}$, then $\alpha_{LB}=\frac{1-a}{\gamma_{3-}}$. Using the argument above, we conclude that there exists a unique $z_0\in\left(\frac{1-a}{\gamma_{3-}},\alpha_{UB}\right)$ such that $\psi(z_0)=0$. The result follows after choosing $z_0=\alpha_{3-}$. Suppose now that $\overline c_2< \frac{N_3N_2}{2N_1}$. This implies that $\alpha_{LB}=\alpha_0$. If $\psi(\alpha_{LB})\leq 0$, then the result follows with $z_0=\alpha_{3-}$. Suppose $\psi(\alpha_{LB})>0$. By the intermediate value theorem and the strict monotonicity of $\psi$, there exists a unique solution $z_1\in\left(\frac{1-a}{\gamma_{3-}},\alpha_{LB}\right)$. Hence, there is no solution on the interval $(\alpha_{LB},\alpha_{UB})$, which completes the proof.
    \end{proof}
    \begin{remark}\label{remark:psi}
        Lemma \ref{lemma:exist alpha3} implies that $\psi(z)=0$ has at most one solution in $(\alpha_{LB},\alpha_{UB})$. Moreover, if $\overline c_2\geq \frac{N_3N_2}{2N_1}$, then the existence of $\alpha_{3-}$ is guaranteed. Since $\psi(\alpha_0)\leq 0$, then by Lemma \ref{lemma:exist alpha3}, $\alpha_{3-}$ is the unique solution to $\psi(z)=0$ on $(\alpha_{LB},\alpha_{UB})$. 
    \end{remark}

    \begin{remark}\label{remark:g is classical soln}
        By construction, $g$ defined in \eqref{eqn:g first} is twice continuously differentiable and satisfies the HJB equation in \eqref{eqn:g first}. As such, $g$ is a classical solution to the HJB equation in \eqref{eqn:hjborig}. By a standard verification lemma, $g$ equals the value function $V$ of the optimization problem in \eqref{eq:V} if Conditions $(i)$ and $(ii)$ in Theorem \ref{prop:w0<u1<u2} hold.
    \end{remark}

	\subsection{Proof of Theorem \ref{prop:u1<w0<u2}}
    In this section, we present the key results used to obtain Theorem \ref{prop:u1<w0<u2}. The discussion serves as the proof for Theorem \ref{prop:u1<w0<u2}. 
    
    Let Assumption \ref{assume:w1<w2} hold. Moreover, suppose that $\overline c_1+\overline c_2\geq \frac{N_3N_2}{2N_1}$ and $\psi(\alpha_{0})> 0$. 

    \subsubsection*{Deriving the analytical solution.}
    
    Suppose for now that $u_1<w_0\leq u_2$. In the region $\{x<u_1\}$, we get \eqref{eqn:region 1} and still obtain its solution $g_1(x)=K_1x^{\gamma_1}$, where $\gamma_1$ is given by \eqref{eqn:gamma1} and $K_1>0$ is an unknown constant. In the region $\{u_1< x< w_0\}$, we have $(C^*_1,C^*_2)=(0,\overline{c}_2)$ and $\theta^*_i$ satisfies \eqref{eqn:optimal_theta} for $i=1,2$. The HJB equation \eqref{eqn:hjborig} then becomes
	\begin{equation}\label{eqn:ode2}
		-M\frac{\left[g'(x)\right]^2}{g''(x)}-\overline{c}_2g'(x)-\beta g(x)+(1-a)\overline c_2=0,
	\end{equation}
	where $M:=\frac{N_1\beta}{N_2-N_1}$. Let $g_2$ be the solution to \eqref{eqn:ode2}. Suppose $g_0$ is the concave solution to $-M\frac{\left[g'(x)\right]^2}{g''(x)}-\overline{c}_2g'(x)-\beta g(x)=0$. As in \cite{hojgaard1999}, the concavity of $g_0$ implies the existence of a function $\chi:\mathbb{R}\to [0,\infty)$ satisfying
	\begin{equation*}
		-\ln [g_0'(\chi(z))]=z.
	\end{equation*}
	It then follows that
	\begin{equation}\label{eqn:odet}
		g_0'(\chi(z))=e^{-z}\quad\mbox{and}\quad g_0''(\chi(z))=-\frac{e^{-z}}{\chi'(z)}.
	\end{equation}
	Substituting $x=\chi(z)$ in \eqref{eqn:ode2} and using \eqref{eqn:odet} yield
	\begin{equation*}
		0=M\chi'(z)e^{-z}-\overline{c}_2e^{-z}-\beta g_0(\chi(z)).
	\end{equation*}
	Differentiating with respect to $z$ and using \eqref{eqn:odet} once more yield
	\begin{equation*}
		0=M\chi''(z)e^{-z}-(M+\beta)\chi'(z)e^{-z}+\overline{c}_2e^{-z},
	\end{equation*}	
	or, equivalently,
	\begin{equation}\label{eqn:odet2}
		0=\chi''(z)-\left(1+\frac{\beta}{M}\right)\chi'(z)+\frac{\overline{c}_2}{M}.
	\end{equation}
	The solution to \eqref{eqn:odet2} is given by
	\begin{equation}
		\chi(z)=k_1e^{\left(1+\frac{\beta}{M}\right)z}+\frac{\overline{c}_2}{M+\beta}z+k_2=k_1e^{\frac{N_2}{N_1}z}+\frac{\overline{c}_2(N_2-N_1)}{N_2\beta}z+k_2,
	\end{equation}
	where $k_1,k_2$ are constants. From \eqref{eqn:odet} and \eqref{eqn:defn of u1 & u2}, it follows that 
        \begin{equation}\label{eqn:chi to u1}
            \chi(-\ln (1-a))=u_1. 
        \end{equation} Moreover, we obtain
	\begin{equation}\label{eqn:X prime}
		\chi'(z)=\frac{k_1N_2}{N_1}e^{\frac{N_2}{N_1}z}+\frac{\overline{c}_2(N_2-N_1)}{N_2\beta}.
	\end{equation} Since $g_0$ is concave, we use \eqref{eqn:odet} to obtain $\chi'(z)>0$. This implies that $k_1\geq 0$ and that $\chi$ is strictly increasing. Combining this with the fact that $\chi$ is continuous, its inverse, denoted by $\chi^{-1}$, exists and is also strictly increasing and continuous. From \eqref{eqn:odet} we determine a solution to \eqref{eqn:ode2} by
	\begin{equation*}
		g_2(x)=\int_{u_1}^xe^{-\chi^{-1}(y)}dy+K_2,
	\end{equation*}
	where $\frac{(1-a)\overline{c}_2}{\beta}$ is already incorporated into the constant $K_2$. Using Lemma \ref{lemma:plateau}, the HJB equation in the region $\{u_1<x<u_2\}$ becomes \eqref{eqn:eqn3hjb}, whose solution is $g_3(x)=K_{3+}e^{\gamma_{3+}x}+K_{3-}e^{\gamma_{3-}x}+\frac{(1-a)\overline c_2}{\beta}$, where $\gamma_{3\pm}$ is given by \eqref{eqn:gamma234} and $K_{3\pm}$ are constants. In the region $\{x>u_2\}$, we obtain \eqref{eqn:hjb4} and its solution $g_4(x)=K_{4-}e^{\gamma_{4-}x}+\frac{a\overline c_1+(1-a)\overline c_2}{\beta}$, which satisfies $\lim_{x\to\infty}g(x)=\frac{a\overline c_1+(1-a)\overline c_2}{\beta}$, where $\gamma_{4-}$ is given by \eqref{eqn:gamma234} and $K_{4-}$ is a constant. We then conjecture the following solution:
	\begin{equation*}
		g(x)=
		\begin{cases}
			K_1x^{\gamma_1}&\mbox{if $x<u_1$,}\\
			\int_{u_1}^xe^{-\chi^{-1}(y)}dy+K_2&\mbox{if $u_1<x<w_0$,}\\
            K_{3+}e^{\gamma_{3+}x}+K_{3-}e^{\gamma_{3-}x}+\frac{(1-a)\overline c_2}{\beta}&\mbox{if $w_0<x<u_2$,}\\
			K_{4-}e^{\gamma_{4-}x}+\frac{a\overline c_1+(1-a)\overline c_2}{\beta}&\mbox{if $x>u_2$,}
		\end{cases}
	\end{equation*}
    where $K_1,K_{3\pm},K_{4-},u_1,u_2$ and $w_0$ are still to be determined. 

    To ensure twice continuous differentiability at $x=u_1$, we obtain the following equations:
	\begin{equation}\label{eqn:system1 case 2}
		\begin{aligned}
			K_1u_1^{\gamma_1}&=K_2\\
			K_1\gamma_1 u_1^{\gamma_1-1}&=e^{-\chi^{-1}(u_1)}=1-a\\
			K_1\gamma_1(\gamma_1-1) u_1^{\gamma_1-2}&=-\frac{e^{-\chi^{-1}(u_1)}}{\chi'(-\ln(1-a))}=-\frac{1-a}{\frac{k_1N_2}{N_1}(1-a)^{-\frac{N_2}{N_1}}+\frac{\overline{c}_2(N_2-N_1)}{N_2\beta}}.
		\end{aligned}
	\end{equation}
	From the first and second equations of \eqref{eqn:system1 case 2}, we obtain
    \begin{equation}\label{eqn:K1 & K2}
        K_1=\frac{(1-a)u_1^{1-\gamma_1}}{\gamma_1}\quad\mbox{and}\quad       K_2=\frac{(1-a)u_1}{\gamma_1}.
    \end{equation}
    Dividing the second equation of \eqref{eqn:system1 case 2} by the third equation yields
	\begin{equation}\label{eqn:u1 - k1}
		u_1=k_1(1-a)^{-\frac{N_2}{N_1}}+\frac{\overline{c}_2N_1(N_2-N_1)}{N_2^2\beta}.
	\end{equation} Using \eqref{eqn:xsoln} and \eqref{eqn:chi to u1} yields
	\begin{equation*}
		u_1=k_1(1-a)^{-\frac{N_2}{N_1}}-\frac{\overline c_2(N_2-N_1)}{N_2\beta}\ln (1-a)+k_2.
	\end{equation*} Comparing both equations for $u_1$ gives us 
    \begin{equation*}
        k_2=\frac{\overline c_2(N_2-N_1)}{N_2\beta}\left(\frac{N_1}{N_2}+\ln(1-a)\right).
    \end{equation*}
	We now determine an expression for $k_1$. From \eqref{eqn:defn of w0}, we obtain
	\begin{equation*}
		1=\frac{(\mu_1\sigma_2-\rho\mu_2\sigma_1)g_2'(w_0)}{(\rho^2-1)\sigma_1^2\sigma_2g_2''(w_0)}=\frac{(\mu_1\sigma_2-\rho\mu_2\sigma_1)\chi'(\chi^{-1}(w_0))}{(1-\rho^2)\sigma_1^2\sigma_2},
	\end{equation*}
	or, equivalently,
	\begin{equation}\label{eqn:X'(X^-1(w0))}
		\chi'(\chi^{-1}(w_0))=\frac{(1-\rho^2)\sigma_1^2\sigma_2}{\mu_1\sigma_2-\rho\mu_2\sigma_1}=\frac{w_1}{1-\gamma_1}=\frac{N_3(N_2-N_1)}{2N_1\beta}.
	\end{equation}
    From \eqref{eqn:X prime}, we obtain
    \begin{equation}\label{eqn:X'(X^-1(w0)) pt2}
        \chi'(\chi^{-1}(w_0))=\frac{k_1N_2}{N_1}e^{\frac{N_2}{N_1}\chi^{-1}(w_0)}+\frac{\overline{c}_2(N_2-N_1)}{N_2\beta}.
    \end{equation}
    Combining \eqref{eqn:X'(X^-1(w0))} and \eqref{eqn:X'(X^-1(w0)) pt2} yields
    \begin{equation*}
        \chi^{-1}(w_0)=\frac{N_1}{N_2}\ln\left[\frac{N_1(N_2-N_1)}{k_1N_2\beta}\left(\frac{N_3}{2N_1}-\frac{\overline c_2}{N_2}\right)\right].
    \end{equation*}
    Since $\psi(\alpha_0)<0$ implies $\overline c_2<\frac{N_3N_2}{2N_1}$, $\chi^{-1}(w_0)$ is well-defined. We then obtain an expression for $k_1$ by
    \begin{equation*}
        k_1=\frac{N_1(N_2-N_1)}{N_2\beta}\left(\frac{N_3}{2N_1}-\frac{\overline c_2}{N_2}\right)e^{-\frac{N_2}{N_1}\chi^{-1}(w_0)}>0.
    \end{equation*}
	
    Let $\alpha_{3+}:=K_{3+}e^{\gamma_{3+}u_1}$ and $\alpha_{3-}:=K_{3-}e^{\gamma_{3-}u_1}$. To ensure twice continuous differentiability at $x=w_0$, we need the following equations to hold:
	\begin{equation}\label{eqn:system2 case 2}
		\begin{aligned}
			e^{-\chi^{-1}(w_0)}&=\gamma_{3+}\alpha_{3+}e^{\gamma_{3+}(w_0-u_1)}+\gamma_{3-}\alpha_{3-}e^{\gamma_{3-}(w_0-u_1)},\\
            -\frac{e^{-\chi^{-1}(w_0)}}{\frac{w_1}{1-\gamma_1}}=-\frac{e^{-\chi^{-1}(w_0)}}{\chi'(\chi^{-1}(w_0))}&=\gamma_{3+}^2\alpha_{3+}e^{\gamma_{3+}(w_0-u_1)}+\gamma_{3-}^2\alpha_{3-}e^{\gamma_{3-}(w_0-u_1)}.
		\end{aligned}
	\end{equation}
    From the first equation, we get $k_1$ by
    \begin{equation*}
        k_1=\frac{N_1(N_2-N_1)}{N_2\beta}\left(\frac{N_3}{2N_1}-\frac{\overline c_2}{N_2}\right)\left[\gamma_{3+}\alpha_{3+}e^{\gamma_{3+}(w_0-u_1)}+\gamma_{3-}\alpha_{3-}e^{\gamma_{3-}(w_0-u_1)}\right]^{\frac{N_2}{N_1}}.
    \end{equation*}
    Dividing the first equation by the second equation in \eqref{eqn:system2 case 2} yields
    \begin{equation*}
        \frac{w_1}{\gamma_1-1}=\frac{\gamma_{3+}\alpha_{3+}e^{\gamma_{3+}(w_0-u_1)}+\gamma_{3-}\alpha_{3-}e^{\gamma_{3-}(w_0-u_1)}}{\gamma_{3+}^2\alpha_{3+}e^{\gamma_{3+}(w_0-u_1)}+\gamma_{3-}^2\alpha_{3-}e^{\gamma_{3-}(w_0-u_1)}},
    \end{equation*}
    which is equivalent to
    \begin{equation}\label{eqn:w0 - u1}
        w_0=u_1+\frac{1}{\gamma_{3+}-\gamma_{3-}}\ln\left[\frac{\gamma_{3-}\alpha_{3-}\left(\frac{w_1}{\gamma_1-1}\gamma_{3-}-1\right)}{\gamma_{3+}\alpha_{3+}\left(1-\frac{w_1}{\gamma_1-1}\gamma_{3+}\right)}\right].
    \end{equation}
    It must be noted that we have not established that the above formula for $w_0$ is well-defined.

     Let $\alpha_{4+}:=K_{4-}e^{\gamma_{4-}u_2}$; we have $\alpha_{4-}=\frac{a}{\gamma_{4-}}$. To ensure twice continuous differentiability at $x=u_2$, we obtain the system of equations in \eqref{eqn:system at u_2} and the expression for $u_2$ in \eqref{eqn:u2}. Combining \eqref{eqn:w0 - u1} and \eqref{eqn:u2} yields 
    \begin{equation}\label{eqn:u2 - w0}
        u_2=w_0+\frac{1}{\gamma_{3+}-\gamma_{3-}}\ln\left[\frac{\left(\frac{w_1}{\gamma_1-1}\gamma_{3-}-1\right)(\gamma_{4-}-\gamma_{3-})}{\left(1-\frac{w_1}{\gamma_1-1}\gamma_{3+}\right)(\gamma_{3+}-\gamma_{4-})}\right].
    \end{equation}
    Thus, we have obtained the form of the candidate value function $g$ defined in \eqref{eqn:candidate 3} and the formulas for $w_0$, $u_1$ and $u_2$ in \eqref{eqn:w0 - u1}, \eqref{eqn:u1 - k1} and \eqref{eqn:u2 - w0}, respectively. Similar to the proof of Theorem \ref{prop:u1<w0<u2}, we have not yet (i) solved for $\alpha_{3-}$, which implies that $w_0$ and $u_2$ may not be well-defined, (ii) guaranteed that $u_1< w_0\leq u_2$, and (iii) shown that $g$ is increasing and concave.

    \subsubsection*{Establishing the bounds for $\alpha_{3-}$.}

    Since $\overline c_2<\frac{N_3N_2}{2N_1}$, it can be shown that $\frac{w_1}{\gamma_1-1}\gamma_{3-}-1>0$, which further implies that the expression for $w_0$ in \eqref{eqn:w0 - u1} is well-defined. To ensure that $w_0>u_1$, we must have
    \begin{equation*}
        \frac{\gamma_{3-}\alpha_{3-}\left(\frac{w_1}{\gamma_1-1}\gamma_{3-}-1\right)}{\gamma_{3+}\alpha_{3+}\left(1-\frac{w_1}{\gamma_1-1}\gamma_{3+}\right)}>1,
    \end{equation*}
    which is equivalent to
    \begin{equation*}
        \alpha_{3-}<\frac{1-a}{\gamma_{3-}(\gamma_{3+}-\gamma_{3-})}\left(\frac{1-\gamma_1}{w_1}+\gamma_{3+}\right)=\alpha_{0}.
    \end{equation*}
    From Lemma \ref{lemma: LB for alpha3}, we have $\alpha_0>\frac{1-a}{\gamma_{3-}}$, which implies that $\alpha_{LB}=\alpha_0$. Together with Lemma \ref{lemma:alpha3-<0}, we have $\alpha_{3-}<0<\alpha_{3+}$. 
    This finding is consistent with the result in Lemma \ref{lemma:w_0<u_1}. 
    Also, the statements regarding $g$ in Remark \ref{remark:g is classical soln} apply here as well.

   For $u_2$ defined in \eqref{eqn:u2 - w0}, we have $u_2\geq w_0$ by $\overline c_1+\overline c_2\geq \frac{N_3N_2}{2N_1}$. Recalling Lemma \ref{lemma:u_1<u_2}, we must have $\alpha_{3-}< \alpha_{UB}$ to ensure that $u_1<u_2$. Hence, we have established that $u_1<w_0\leq u_2$. 
   
   Since $\overline c_1+\overline c_2\geq \frac{N_3N_2}{2N_1}$, by Lemma \ref{lemma:LB<UB}, $\alpha_0\leq \alpha_{UB}$ holds. By following similar arguments as in the proof of Theorem \ref{prop:u1<w0<u2}, we obtain the following bounds for $\alpha_{3-}$:
    \begin{equation}
        \frac{1-a}{\gamma_{3-}}<\alpha_{3-}<\alpha_{0}.
    \end{equation}

    \subsubsection*{Solving for $\alpha_{3-}$.}
    Since $\psi(\alpha_0)>0$ by assumption and $\psi\left(\frac{1-a}{\gamma_{3-}}\right)<0$, $\alpha_{3-}$ is the unique solution of $\psi(z)=0$ on $\left(\frac{1-a}{\gamma_{3-}},\alpha_{0}\right)$ via Lemma \ref{lemma:exist alpha3}.  
    
    \begin{remark}
        The scenario $u_1\leq u_2\leq w_0$ is not possible under the assumptions $\overline c_1+\overline c_2\geq \frac{N_3N_2}{2N_1}$ and $\psi(\alpha_{0})> 0$. To see this, suppose $u_1\leq u_2\leq w_0$ holds. We know that in the region $\{x>w_0\}$, we have $g_4(x)=\frac{a}{\gamma_{4-}}e^{\gamma_{4-}(x-u_2)}+\frac{a\overline c_1+(1-a)\overline c_2}{\beta}$. By definition of $w_0$, we must have $\frac{w_1}{\gamma_1-1}=\frac{1}{\gamma_{4-}}$, which holds if and only if $\overline c_1+\overline c_1=\frac{N_3N_2}{2N_1}$. In the region $\{u_2<x<w_0\}$, we have $g_3(x)=\int_{u_2}^x e^{-Z^{-1}(y)}dy+K_3$, where $Z(y)=k_3e^{\frac{N_2}{N_1}y}+\frac{(\overline c_1+\overline c_2)(N_2-N_1)}{N_2\beta}y+k_4$. Using the definition of $w_0$ once more yields $\frac{w_1}{1-\gamma_1}=Z'(Z^{-1}(w_0))=\frac{k_3N_2}{N_1}e^{\frac{N_2}{N_1}y}+\frac{(\overline c_1+\overline c_2)(N_2-N_1)}{N_2\beta}$. From $\overline c_1+\overline c_1=\frac{N_3N_2}{2N_1}$, $\frac{w_1}{1-\gamma_1}=\frac{(\overline c_1+\overline c_2)(N_2-N_1)}{N_2\beta}$, which implies that $\frac{k_3N_2}{N_1}e^{\frac{N_2}{N_1}y}=0$, or, equivalently, $k_3=0$. The reinsurance level for Line 1 when $u_2<x<w_0$ must then satisfy $1-\theta_1(x)=\frac{1-\gamma_1}{w_1}Z'(Z^{-1}(x))=1$, which contradicts the definition of $w_0$. 
    \end{remark}
    
    \subsection{Proof of Theorem \ref{prop:u1<u2<w0}}

    Let Assumption \ref{assume:w1<w2} hold and suppose that $\overline c_1+\overline c_2<\frac{N_3N_2}{2N_1}$. 

    \subsubsection*{Proving that $w_0$ is infinite.}
    Suppose that $w_0$ exists and satisfies $u_1\leq u_2< w_0$. In the region $\{x>w_0\}$, we get \eqref{eqn:hjb4}, and its solution is $g_4(x)=\frac{a}{\gamma_{4-}}e^{\gamma_{4-}x}+\frac{a\overline c_1+(1-a)\overline c_2}{\beta}$, where $\gamma_{4-}$ is given in \eqref{eqn:gamma234} and $K_{4-}$ is a constant. Let $g_3(x)$ be the solution to the HJB equation in the region $\{u_2<x<w_0\}$. From \eqref{eqn:defn of w0} and the assumption that $\overline c_1+\overline c_2<\frac{N_3N_2}{2N_1}$, 
	\begin{equation*}
		1=\frac{(\mu_1\sigma_2-\rho\mu_2\sigma_1)g_4'(w_0)}{(\rho^2-1)\sigma_1^2\sigma_2g_4''(w_0)}=\frac{-2N_1\beta}{N_3(N_2-N_1)\gamma_{4-}}<1,
	\end{equation*}
	which is a contradiction. This suggests that no such $w_0$ exists (or we write $w_0=\infty$), which implies that there are only two switching points, $u_1$ and $u_2$.

    \subsubsection*{Deriving the analytical solution.}
    
	Suppose for now that $u_1\leq  u_2<w_0=\infty$. In the region $\{x<u_1\}$, we have \eqref{eqn:region 1} and obtain its solution $g_1(x)=K_1x^{\gamma_1}$, where $K_1>0$ is a constant, and $\gamma_1$ is given by \eqref{eqn:gamma1}. In the region $\{u_1<x<u_2\}$, we obtain $g_2(x)=\int_{u_1}^xe^{-\chi^{-1}(y)}dy+K_2$, where $\chi(z)$ satisfies \eqref{eqn:xsoln}, and $K_2$ is a constant. From \eqref{eqn:odet} and \eqref{eqn:defn of u1 & u2}, it follows that $\chi(-\ln (1-a))=u_1$ and $\chi(-\ln a)=u_2$. Hence, from \eqref{eqn:xsoln}, we obtain $k_1$ and $k_2$ in \eqref{eqn:k1 & k2 case 3}.
	
    In the region $\{x > u_2\}$, we have $(C^*_1,C^*_2)=(\overline{c}_1,\overline{c}_2)$ and the HJB equation becomes
	\begin{equation}\label{eqn:reg3x}
		\mathcal{H}(x)-(\overline c_1+\overline c_2)g'(x)-\beta g(x)+a\overline c_1+(1-a)\overline{c}_2=0,
	\end{equation}
    where $\mathcal{H}$ is defined by \eqref{eqn:optim reinsurance}.
	The first order conditions for \eqref{eqn:reg3x} still satisfy \eqref{eqn:optimal_theta}. Substituting \eqref{eqn:optimal_theta} into \eqref{eqn:reg3x} yields
	\begin{equation}\label{eqn:ode3}
		0=-M\frac{\left[g'(x)\right]^2}{g''(x)}-(\overline c_1+\overline{c}_2)g'(x)-\beta g(x)+a\overline{c}_1+(1-a)\overline{c}_2,
	\end{equation}
	where $M:=\frac{N_1\beta}{N_2-N_1}$. A solution to \eqref{eqn:ode3} that satisfies the condition $\lim_{x\to\infty}g(x)=\frac{a\overline c_1+(1-a)\overline c_2}{\beta}$ is given by
	\begin{equation*}
		g_3(x)=\frac{K_3}{\gamma_3}e^{\gamma_3 x}+\frac{a\overline c_1+(1-a)\overline c_2}{\beta},
	\end{equation*}
	where $\gamma_3$ is given by \eqref{eqn:gamma3}. We then conjecture the following solution:
	\begin{equation*}
		g(x)=
		\begin{cases}
			K_1x^{\gamma_1}&\mbox{if $x<u_1$,}\\
			\int_{u_1}^xe^{-\chi^{-1}(y)}dy+K_2&\mbox{if $u_1<x<u_2$,}\\
			\frac{K_3}{\gamma_3}e^{\gamma_3 x}+\frac{a\overline c_1+(1-a)\overline c_2}{\beta}&\mbox{if $x>u_2$,}
		\end{cases}
	\end{equation*}
	where $K_1,K_2,K_3,u_1,u_2$ are still to be determined.  To ensure that $g$ is twice continuously differentiable, we require $g$, $g'$, and $g''$ to be continuous at the switching points $u_1$ and $u_2$. At $x=u_1$, we obtain the system of equations in \eqref{eqn:system1 case 2}, and $K_1$ and $K_2$ in \eqref{eqn:K1 & K2}. Dividing the second equation by the third equation in \eqref{eqn:system1 case 2} yields
	\begin{equation}\label{eqn:u1u2}
		\frac{u_1}{1-\gamma_1}=\frac{u_2-u_1+\frac{\overline{c}_2(N_2-N_1)}{N_2\beta}\ln \left(\frac{a}{1-a}\right)}{a^{-\frac{N_2}{N_1}}-(1-a)^{-\frac{N_2}{N_1}}}\cdot\frac{N_2}{N_1}(1-a)^{-\frac{N_2}{N_1}}+\frac{\overline c_2(N_2-N_1)}{N_2\beta}.
	\end{equation}
	Here, $u_1$ and $u_2$ are still to be determined.
	
	Since in the neighborhood of $u_2$ the function $g$ satisfies \eqref{eqn:ode3} and \eqref{eqn:reg3x} with $\theta^*_1,\theta^*_2$ satisfying \eqref{eqn:optimal_theta}, it suffices to show that $g'$ and $g''$ are continuous at $u_2$. To see this, we have from \eqref{eqn:ode3} and \eqref{eqn:reg3x} the following equation:
	\small\begin{equation*}
		-M\frac{\left[g_2'(u_2)\right]^2}{g_2''(u_2)}-\overline{c}_2g_2'(u_2)-\beta g_2(u_2)+(1-a)\overline c_2=-M\frac{\left[g_3'(u_2)\right]^2}{g_3''(u_2)}-(\overline c_1+\overline{c}_2)g'_3(u_2)-\beta g_3(u_2)+a\overline{c}_1+(1-a)\overline{c}_2.
	\end{equation*}\normalsize
	If $g'_2(u_2)=g'_3(u_2)$ and $g''_2(u_2)=g''_3(u_2)$, then the above equation can be rewritten as:
	\begin{equation*}
		-\beta g_2(u_2)=-\overline c_1g'_3(u_2)-\beta g_3(u_2)+a\overline{c}_1.
	\end{equation*}
	Because $g'_3(u_2)=a$, we get $g_2(u_2)=g_3(u_2)$. At $x=u_2$, we then have the following system of equations:
	\begin{equation*}
		\begin{aligned}
			a=e^{-\chi^{-1}(u_2)}&=K_3e^{\gamma_3u_2}\\
			-\frac{a}{\frac{k_1N_2}{N_1}\cdot a^{-\frac{N_2}{N_1}}+\frac{\overline c_2(N_2-N_1)}{N_2\beta}}=-\frac{e^{-\chi^{-1}(u_2)}}{\chi'(\chi^{-1}(u_2))}&=K_3\gamma_3e^{\gamma_3u_2}.
		\end{aligned}
	\end{equation*}
	The first equation is equivalent to 
	\begin{equation*}
		K_3=ae^{-\gamma_3u_2}.
	\end{equation*}
	Dividing the first equation by the second equation yields:
	\begin{equation}\label{eqn:u2u1}
		u_2=u_1+\frac{\overline c_1N_1(N_2-N_1)}{N_2^2\beta}\left(1-\left(\frac{a}{1-a}\right)^{\frac{N_2}{N_1}}\right)-\frac{\overline c_2(N_2-N_1)}{N_2\beta}\ln \left(\frac{a}{1-a}\right).
	\end{equation}
	Since $a\leq \frac{1}{2}$, it is clear that $u_2\geq u_1$. Using \eqref{eqn:u1u2} and \eqref{eqn:u2u1} yields
	\begin{equation}\label{eqn:u1solved}
		u_1=(1-\gamma_1)\left[\frac{\overline c_1(N_2-N_1)}{N_2\beta}\left(\frac{a}{1-a}\right)^{\frac{N_2}{N_1}}+\frac{\overline c_2(N_2-N_1)}{N_2\beta}\right].
	\end{equation}
	Thus, we obtain $g$ defined in \eqref{eqn:g w0>u1}.

    \subsubsection*{Showing no ``zero reinsurance".}
	We now prove that $\theta_1^*(x)>0$ and $\theta_2^*(x)>0$ for any $x>0$. Since $\chi^{-1}(x)$ is an increasing function on $[u_1,u_2]$ and $\overline c_1+\overline c_2<\frac{N_3N_2}{2N_1}$ by assumption, in the region $\{u_1<x<u_2\}$,
	\begin{equation*}
		\begin{aligned}
			1-\theta_1^*(x)=\frac{1-\gamma_1}{w_1}\chi'(\chi^{-1}(x))\leq \frac{1-\gamma_1}{w_1}\chi'(\chi^{-1}(u_2))=\frac{2N_1}{N_3N_2}(\overline c_1+\overline c_2)<1,
		\end{aligned}
	\end{equation*}
	or, equivalently,
	\begin{equation*}
		\theta_1^*(x)=1-\frac{1-\gamma_1}{w_1}\chi'(\chi^{-1}(x))>0\quad\mbox{and}\quad \theta_2^*(x)=1-\frac{1-\gamma_1}{w_2}\chi'(\chi^{-1}(x))>0.
	\end{equation*}
	Consequently, in the region $\{x<u_1\}$, we have the following:
	\begin{equation*}
		\theta_1^*(x)=1-\frac{x}{w_1}>0\quad\mbox{and}\quad \theta_2^*(x)=1-\frac{x}{w_2}>0.
	\end{equation*}
	In the region $\{x>u_2\}$, since $\overline c_1+\overline c_2<\frac{N_3N_2}{2N_1}$ by assumption, we have
	\begin{equation*}
		1-\theta_1^*(x)=-\frac{1-\gamma_1}{w_1\gamma_3}=-\frac{1-\gamma_1}{w_1}\cdot\frac{-(\overline c_1+\overline c_2)(N_2-N_1)}{N_2\beta}=\frac{2N_1}{N_3N_2}(\overline c_1+\overline c_2)<1,
	\end{equation*}
	which, with the assumption that $w_1\leq w_2$, is equivalent to
	\begin{equation*}
		1-\theta_1^*(x)=-\frac{1-\gamma_1}{w_1\gamma_3}<1\quad\mbox{and}\quad 1-\theta_2^*(x)=-\frac{1-\gamma_1}{w_2\gamma_3}<1.
	\end{equation*}

\section{Conclusion}\label{sec:conclusion}
We study a bivariate optimal dividend problem where an insurer maximizes the expected weighted sum of total dividends of two collaborating business lines under the diffusion model with correlated Brownian motions. 
In addition to dividend payout, our model allows the manager of the insurer to purchase proportional reinsurance contracts to mitigate the risk exposure of each line and to inject capital from one line into the other to prevent potential bankruptcy. We obtain a complete analytical solution to this problem; in particular, we identify three scenarios and obtain the value function and optimal strategies in closed form, respectively. 

We show that the optimal dividend strategy is a threshold strategy, and the more important business line has a smaller threshold than the less important line. 
 The optimal reinsurance coverage strategy is shown to be decreasing with respect to the aggregate reserve level. The reinsurance coverage of both lines remain constant (independent of the aggregate reserve level) as soon as the aggregate reserve level hits the switching point that affect reinsurance. The correlation coefficient also plays a significant role in determining the optimal reinsurance coverage. The optimal capital transfer strategy (stated in Theorem \ref{thm:op_L}) is consistent in all three scenarios, and the decision is to either transfer reserves to save the line at risk or wait until the surplus pair leaves the current region.

	\bibliographystyle{plainnat} 
	\bibliography{References}

\appendix
	
\section{Proof of Proposition \ref{prop:hjb}}
\label{app:hjb}

	\begin{proof}
	Let $(0,h]$ be a small interval and $\epsilon>0$. Suppose that for each surplus $X_1(h),X_2(h)>0$, there exists an admissible strategy $u_{\epsilon}$ such that
	\begin{equation}
		J(X_1(h),X_2(h);u_{\epsilon})>V(X_1(h),X_2(h))-\epsilon.
	\end{equation}
	Fix $0\leq c_1\leq \overline c_1$, $0\leq c_2\leq \overline c_2$, and $h>0$. Consider
	\begin{equation*}
		(C_1(t),C_2(t))=
		\begin{cases}
			(c_1,c_2),&\mbox{if $0\leq t\leq \tau \wedge h$}\\
			(C^{X_1(h)}_1(t-h),C^{X_2(h)}_2(t-h)),&\mbox{if $t>h$ and $\tau>h$},
		\end{cases}
	\end{equation*}
	where $C^{X_i(h)}_i(\cdot)$ is the $\epsilon$-optimal process corresponding to the reserve level $X_i(h)$ at time $h$ for $i=1,2$. Define $u:=(\theta_1,\theta_2,C_1,C_2,L_1,L_2)\in\Uc$ and $u_{\epsilon}:=\left(\theta_1,\theta_2,C^{X_1(h)}_1,C^{X_2(h)}_2,L_1,L_2\right)$. Then,
	{\small\begin{equation*}
		\begin{aligned}
			V(x_1,x_2)
			&\geq J(x_1,x_2;u)\\
			&=\mathbb{E}_{(x_1,x_2)}\left[a\int_0^{\tau\wedge h}e^{-\beta t}c_1dt+(1-a)\int_0^{\tau\wedge h}e^{-\beta t}c_2dt\right]\\
			&\quad+\mathbb{E}_{(x_1,x_2)}\left[\mathds{1}_{\{\tau>h\}}\mathbb{E}\left[a\int_h^{\tau}e^{-\beta t}C^{X_1(h)}_1(t-h)dt+(1-a)\int_h^{\tau}e^{-\beta t}C^{X_2(h)}_2(t-h)dt\Big|\mathcal{F}_h\right]\right]\\
			&=[ac_1+(1-a)c_2]\frac{1-\mathbb{E}_{(x_1,x_2)}\left[e^{-\beta(\tau\wedge h)}\right]}{\beta}\\
			&\quad+e^{-\beta h}\mathbb{E}_{(x_1,x_2)}\left[\mathds{1}_{\{\tau>h\}}\mathbb{E}_{(X_1(h),X_2(h))}\left[a\int_0^{\tau-h}e^{-\beta s}C^{X_1(h)}_1(s)ds+(1-a)\int_0^{\tau-h}e^{-\beta s}C^{X_2(h)}_2(s)dt\right]\right]\\
			&=[ac_1+(1-a)c_2]\frac{1-\mathbb{E}_{(x_1,x_2)}\left[e^{-\beta(\tau\wedge h)}\right]}{\beta}+e^{-\beta h}\mathbb{E}_{(x_1,x_2)}\left[\mathds{1}_{\{\tau>h\}}J(X_1(h),X_2(h);u_{\epsilon})\right]\\
			&\geq [ac_1+(1-a)c_2]\frac{1-\mathbb{E}_{(x_1,x_2)}\left[e^{-\beta(\tau\wedge h)}\right]}{\beta}+e^{-\beta h}\mathbb{E}_{(x_1,x_2)}\left[\mathds{1}_{\{\tau>h\}}\left[V(X_1(h),X_2(h))-\epsilon\right]\right]\\
			&\geq [ac_1+(1-a)c_2]\frac{1-\mathbb{E}_{(x_1,x_2)}\left[e^{-\beta(\tau\wedge h)}\right]}{\beta}+e^{-\beta h}\mathbb{E}_{(x_1,x_2)}\left[V(X_1(\tau\wedge h),X_2(\tau\wedge h))\right]-\epsilon.
		\end{aligned}
	\end{equation*}}
	Since $\epsilon$ is arbitrary, then
	\begin{equation}\label{eqn:hjb1}
		\mathbb{E}_{(x_1,x_2)}[V(x_1,x_2)]\geq [ac_1+(1-a)c_2]\frac{1-\mathbb{E}_{(x_1,x_2)}\left[e^{-\beta(\tau\wedge h)}\right]}{\beta}+e^{-\beta h}\mathbb{E}_{(x_1,x_2)}\left[V(X_1(\tau\wedge h),X_2(\tau\wedge h))\right].
	\end{equation}
	Suppose that $V$ is twice continuously differentiable. By It\^{o}'s formula,
	\begin{equation*}
		\begin{aligned}
			V(X_1(\tau\wedge h),X_2(\tau\wedge h))
			&=V(x_1,x_2)+\int_0^{\tau\wedge h}\sum_{i=1}^2\left[\left[\left(\tilde{\kappa}_i-\kappa_i\theta_i\right)\mu_i-c_i\right]\frac{\partial V}{\partial x_i}+\frac{1}{2}\sigma_i^2(1-\theta_i)^2\frac{\partial^2 V}{\partial x_i^2}\right]dt\\
			&\quad+\int_0^{\tau\wedge h}\rho \sigma_1\sigma_2(1-\theta_1)(1-\theta_2)\frac{\partial^2 V}{\partial x_1\partial x_2}dt - \int_0^{\tau\wedge h}\sum_{i=1}^2(1-\theta_i)\sigma_i\frac{\partial V}{\partial x_i}dW_i(t)\\
			&\quad+\sum_{i=1}^2\int_0^{\tau\wedge h} \left[\frac{\partial V}{\partial x_i}(X_1(t-),X_2(t-))-\frac{\partial V}{\partial x_{3-i}}(X_1(t-),X_2(t-))\right]dL^c_i(t)\\
			&\quad+\sum_{X_1(t-)\neq X_1(t),X_2(t-)\neq X_2(t),t\leq (\tau\wedge h)}\left[V(X_1(t),X_2(t))-V(X_1(t-),X_2(t-))\right],
		\end{aligned}			
	\end{equation*}
	where $L_i^c$ is the continuous part of $L_i$. It can be shown that the process\\ $\left\{\int_0^{t}\sum_{i=1}^2(1-\theta_i(s))\sigma_i\frac{\partial V}{\partial x_i}(X_1,X_2)dW_i(s)\right\}_{t\geq 0}$ is a true martingale. Write $\Delta X(t):=X_1(t)-X_1(t-)=X_2(t-)-X_2(t)$ for $t>0$. This implies that $\Delta X(t)=L_1(t)-L_1(t-)$ if $\Delta X(t)>0$ and $\Delta X(t)=L_2(t)-L_2(t-)$ if $\Delta X(t)<0$. Taking expectations and using \eqref{eqn:hjb1} yield
	
		\begin{align}
			0&\geq\frac{e^{-\beta h}}{h}\mathbb{E}_{(x_1,x_2)}\Bigg\{\int_0^{\tau\wedge h}\sum_{i=1}^2\left[\left[\left(\tilde{\kappa}_i-\kappa_i\theta_i\right)\mu_i-c_i\right]\frac{\partial V}{\partial x_i}+\frac{1}{2}\sigma_i^2(1-\theta_i)^2\frac{\partial^2 V}{\partial x_i^2}\right]dt\\
			&\quad+\int_0^{\tau\wedge h}\rho \sigma_1\sigma_2(1-\theta_1)(1-\theta_2)\frac{\partial^2 V}{\partial x_1\partial x_2}dt\\
			&\quad+\sum_{i=1}^2\int_0^{\tau\wedge h} \left[\frac{\partial V}{\partial x_i}(X_1(t-),X_2(t-))-\frac{\partial V}{\partial x_{3-i}}(X_1(t-),X_2(t-))\right]dL^c_i(t)\\
			&\quad+\sum_{\Delta X(t)\neq 0,t\leq (\tau\wedge h)}\left[V(X_1(t-)+\Delta X(t),X_2(t-)-\Delta X(t))-V(X_1(t-),X_2(t-))\right]\Bigg\}\\
			&\quad+ [ac_1+(1-a)c_2]\frac{1-\mathbb{E}_{(x_1,x_2)}\left[e^{-\beta(\tau\wedge h)}\right]}{\beta h}-\frac{1-e^{-\beta h}}{h}V(x_1,x_2).
		\end{align}
	Assuming limits and expectations can be interchanged, then letting $h \to 0$ yields
	\begin{equation*}
		\begin{aligned}
			0&\geq \sum_{i=1}^2\left[\left(\tilde{\kappa}_i-\kappa_i\theta_i\right)\mu_i-c_i\right]\frac{\partial V}{\partial x_i}+\frac{1}{2}\sigma_i^2(1-\theta_i)^2\frac{\partial^2 V}{\partial x_i^2}+\rho \sigma_1\sigma_2(1-\theta_1)(1-\theta_2)\frac{\partial^2 V}{\partial x_1\partial x_2}\\
			&\quad +ac_1+(1-a)c_2-\beta V(x_1,x_2) + \sum_{i=1}^2\left[\frac{\partial V}{\partial x_i}(x_1,x_2)-\frac{\partial V}{\partial x_{3-i}}(x_1,x_2)\right]\Delta x\\
			&\quad + V(x_1+\Delta x,x_2-\Delta x)-V(x_1,x_2),
		\end{aligned}
	\end{equation*}
	where $\Delta x:=X_1(0)=-X_2(0)$. The inequality must hold for any $u\in\Uc$, that is,,
	{\small\begin{align*}
		0\geq \sup_{u\in\Uc}\left\{\Lc^{\vartheta}(V)(x_1,x_2)+\sum_{i=1}^2\left[\frac{\partial V}{\partial x_i}(x_1,x_2)-\frac{\partial V}{\partial x_{3-i}}(x_1,x_2)\right]\Delta x+V(x_1+\Delta x,x_2-\Delta x)-V(x_1,x_2)\right\}.
	\end{align*}}
	Suppose $\frac{\partial V}{\partial x_1}(x_1,x_2)\neq \frac{\partial V}{\partial x_{2}}(x_1,x_2)$. If $\frac{\partial V}{\partial x_i}(x_1,x_2)> \frac{\partial V}{\partial x_{3-i}}(x_1,x_2)$, then $\Delta x$ can be made large enough by transferring an infinite amount of capital to Line $i$, which will make the maximization problem above infeasible. Thus, we must have $\frac{\partial V}{\partial x_1}(x_1,x_2)= \frac{\partial V}{\partial x_{2}}(x_1,x_2)$. This implies that $V(x_1+\Delta x,x_2-\Delta x)=V(x_1,x_2)$ for $\Delta x\in\mathbb{R}$.
	
	Suppose there exists an optimal control strategy $u$ such that $\lim_{t\downarrow 0}u(t)=u(0)$. Then, similarly, we have
	{\small\begin{align*}
		0= \sup_{u\in\Uc}\left\{\Lc^{\vartheta}(V)(x_1,x_2)+\sum_{i=1}^2\left[\frac{\partial V}{\partial x_i}(x_1,x_2)-\frac{\partial V}{\partial x_{3-i}}(x_1,x_2)\right]\Delta x+V(x_1+\Delta x,x_2-\Delta x)-V(x_1,x_2)\right\},
	\end{align*}}
	which completes the proof.
\end{proof}
	
\section{Proof of Lemma \ref{lemma:plateau}}\label{app:plateau}
    \begin{proof}
		We solve \eqref{eqn:optim reinsurance} at the (lower) boundary of $\theta_1$ using the Karush-Kuhn-Tucker (KKT) conditions. Define
		\begin{equation}\label{eqn:Lagrange}\small
			\begin{aligned}
				L(\theta_1,\theta_2,\lambda_1,\lambda_2,\lambda_3,\lambda_4)&=\sum_{i=1}^2\left[\left(1-\theta_i\right)\mu_ig'(x)+\frac{1}{2}\sigma_i^2(1-\theta_i)^2g''(x)\right]+\rho\sigma_1\sigma_2(1-\theta_1)(1-\theta_2)g''(x)-\beta g(x)\\
				&\quad+\lambda_1\theta_1+\lambda_2\theta_2+\lambda_3(1-\theta_1)+\lambda_4(1-\theta_2).
			\end{aligned}
		\end{equation}\normalsize
		The associated KKT conditions are summarized below:
		\begin{equation}\label{eqn:kkt}
			\begin{cases}
				-\mu_1g'(x)-\sigma_1^2(1-\theta_1)g''(x)-\rho\sigma_1\sigma_2(1-\theta_2)g''(x)+\lambda_1-\lambda_3=0\\
				-\mu_2g'(x)-\sigma_2^2(1-\theta_2)g''(x)-\rho\sigma_1\sigma_2(1-\theta_1)g''(x)+\lambda_2-\lambda_4=0\\
				\lambda_1\theta_1=0,\quad\lambda_2\theta_2=0,\quad\lambda_3(1-\theta_1)=0,\quad\lambda_4(1-\theta_2)=0\\
				\theta_1,\theta_2\in[0,1],\quad \lambda_1,\lambda_2,\lambda_3,\lambda_4\geq 0.
			\end{cases}
		\end{equation}
		At the boundary $\theta_1=0$, it follows from the complementary slackness conditions that $\lambda_1\geq 0$ and $\lambda_3=0$. The first two equations of \eqref{eqn:kkt} can then be rewritten as
		\begin{equation}\label{eqn:foc}
			\begin{aligned}
				-\mu_1g'(x)-\sigma_1^2g''(x)-\rho\sigma_1\sigma_2(1-\theta_2)g''(x)+\lambda_1&=0\\
				-\mu_2g'(x)-\sigma_2^2(1-\theta_2)g''(x)-\rho\sigma_1\sigma_2g''(x)+\lambda_2-\lambda_4&=0.
			\end{aligned}
		\end{equation}
		We first consider the case $\lambda_1=0$. Then,
		\begin{equation}\label{eqn:l1=0}
			\begin{aligned}
				1-\theta_2&=-\frac{\mu_1g'(x)}{\rho\sigma_1\sigma_2g''(x)}-\frac{\sigma_1}{\rho\sigma_2}\\
				\lambda_4-\lambda_2&=\frac{\mu_1\sigma_2-\rho\mu_2\sigma_1}{\rho\sigma_1}g'(x)+\frac{\sigma_1\sigma_2(1-\rho^2)}{\rho}g''(x).
			\end{aligned}
		\end{equation}
		If $\lambda_2> 0$ and $\lambda_4=0$, then $\theta_2=0$ and \eqref{eqn:l1=0} can be rewritten as:
		\begin{equation*}
			\begin{aligned}
				1&=-\frac{\mu_1g'(x)}{\rho\sigma_1\sigma_2g''(x)}-\frac{\sigma_1}{\rho\sigma_2}\\
				-\lambda_2&=\frac{\mu_1\sigma_2-\rho\mu_2\sigma_1}{\rho\sigma_1}g'(x)+\frac{\sigma_1\sigma_2(1-\rho^2)}{\rho}g''(x).
			\end{aligned}
		\end{equation*}
		Combining the two equations yields:
		\begin{equation*}
			-\lambda_2=\left[\sigma_1(\mu_2\sigma_1-\rho\mu_1\sigma_2)-\sigma_2(\mu_1\sigma_2-\rho\mu_2\sigma_1)\right]\frac{g''(x)}{\mu_1}.
		\end{equation*}
		From Assumption \ref{assume:w1<w2}, we obtain $\sigma_1(\mu_2\sigma_1-\rho\mu_1\sigma_2)\leq\sigma_2(\mu_1\sigma_2-\rho\mu_2\sigma_1)$. Since $g$ is concave, then $g''\leq0$. Hence,
		\begin{equation*}
			0> -\lambda_2=\left[\sigma_1(\mu_2\sigma_1-\rho\mu_1\sigma_2)-\sigma_2(\mu_1\sigma_2-\rho\mu_2\sigma_1)\right]\frac{g''(x)}{\mu_1}\geq 0,
		\end{equation*}
		which is a contradiction. If $\lambda_2= 0$ and $\lambda_4>0$, then $\theta_2=1$ and \eqref{eqn:l1=0} can be rewritten as:
		\begin{equation*}
			\begin{aligned}
				0&=-\frac{\mu_1g'(x)}{\rho\sigma_1\sigma_2g''(x)}-\frac{\sigma_1}{\rho\sigma_2}\\
				\lambda_4&=\frac{\mu_1\sigma_2-\rho\mu_2\sigma_1}{\rho\sigma_1}g'(x)+\frac{\sigma_1\sigma_2(1-\rho^2)}{\rho}g''(x).
			\end{aligned}
		\end{equation*}
		Combining the two equations yields:
		\begin{equation*}
			0<\lambda_4=\frac{\sigma_1}{\mu_1}(\mu_2\sigma_1-\rho\mu_1\sigma_2)g''(x)\leq 0,
		\end{equation*}
		which is also a contradiction. If $\lambda_2>0$ and $\lambda_4>0$ then it must hold that $\theta_2=1$ and $\theta_2=0$, which is a contradiction. If $\lambda_2=0$ and $\lambda_4=0$, then $\theta_2\in[0,1]$. From the second equation in \eqref{eqn:l1=0}, we obtain:
		\begin{equation*}
			\frac{g'(x)}{g''(x)}=-\frac{\sigma_1^2\sigma_2(1-\rho^2)}{\mu_1\sigma_2-\rho\mu_2\sigma_1}=\frac{w_1}{\gamma_1-1}.
		\end{equation*}
		Substituting this to the first equation in \eqref{eqn:l1=0} yields
		\begin{equation*}
			1-\theta_2=\frac{\sigma_1(\mu_2\sigma_1-\rho\mu_1\sigma_2)}{\sigma_2(\mu_1\sigma_2-\rho\mu_2\sigma_1)}=\frac{w_0}{w_2}\in(0,1),
		\end{equation*}
		which is a candidate solution for $\theta_2$.
		
		The second case to consider is $\lambda_1>0$. If $\lambda_2>0$ and $\lambda_4=0$, then $\theta_2=0$. We then obtain from \eqref{eqn:foc} the following equations:
		\begin{equation*}
			\begin{aligned}
				0<\lambda_1&=\mu_1g'(x)+(\sigma_1^2+\rho\sigma_1\sigma_2)g''(x)\\
				0<\lambda_2&=\mu_2g'(x)+(\sigma_2^2+\rho\sigma_1\sigma_2)g''(x).
			\end{aligned}
		\end{equation*}
		Since Assumption \ref{assume:w1<w2} holds, then $\frac{\sigma_1}{\mu_1}(\sigma_1+\rho\sigma_2)\leq\frac{\sigma_2}{\mu_2}(\sigma_2+\rho\sigma_1)$. Hence, the above inequalities imply that
		\begin{equation*}
			\frac{g'(x)}{g''(x)}<\min\left\{-\frac{\sigma_1}{\mu_1}(\sigma_1+\rho\sigma_2),-\frac{\sigma_2}{\mu_2}(\sigma_2+\rho\sigma_1)\right\}=-\frac{\sigma_2}{\mu_2}(\sigma_2+\rho\sigma_1).
		\end{equation*}
		Since we want $g$ to be twice continuously differentiable, then at $x=w_0$, we must have
		\begin{equation*}
			\frac{-(1-\rho^2)\sigma_1^2\sigma_2}{\mu_1\sigma_2-\rho\mu_2\sigma_1}=\frac{w_0}{\gamma_1-1}=\frac{g'(w_0)}{g''(w_0)}<-\frac{\sigma_2}{\mu_2}(\sigma_2+\rho\sigma_1),
		\end{equation*}
		which is equivalent to $\sigma_1(\mu_2\sigma_1-\rho\mu_1\sigma_2)>\sigma_2(\mu_1\sigma_2-\rho\mu_2\sigma_1)$. This contradicts Assumption \ref{assume:w1<w2}. If $\lambda_2=0$ and $\lambda_4>0$, then $\theta_2=1$. We then obtain from \eqref{eqn:foc} the following equations:
		\begin{equation*}
			\begin{aligned}
				0<\lambda_1&=\mu_1g'(x)+\sigma_1^2g''(x)\\
				0<\lambda_4&=-\mu_2g'(x)-\rho\sigma_1\sigma_2g''(x),
			\end{aligned}
		\end{equation*}
		which implies that $-\frac{\rho\sigma_1\sigma_2}{\mu_2}<\frac{g'(x)}{g''(x)}<-\frac{\sigma^2_1}{\mu_1}$. However, this requires that $\mu_2\sigma_1-\rho\mu_1\sigma_2<0$ holds, which is a contradiction. If $\lambda_2>0$ and $\lambda_4>0$, then the complementary slackness conditions will also lead to a contradiction. If $\lambda_2=0$ and $\lambda_4=0$, then we obtain the following equations:
		\begin{equation*}
			\begin{aligned}
				1-\theta_2&=-\frac{\mu_2g'(x)}{\sigma_2^2g''(x)}-\frac{\rho\sigma_1}{\sigma_2}\\
				0<\lambda_1&=\sigma_1^2(1-\rho^2)g''(x)+\frac{\mu_1\sigma_2-\rho\mu_2\sigma_1}{\sigma_2}g'(x).
			\end{aligned}
		\end{equation*}
		The inequality at $x=w_0$ leads to
		\begin{equation*}
			\frac{w_0}{\gamma_1-1}=\frac{g'(w_0)}{g''(w_0)}<-\frac{\sigma_1^2\sigma_2(1-\rho^2)}{\mu_1\sigma_2-\rho\mu_2\sigma_1}=\frac{w_0}{\gamma_1-1},
		\end{equation*}
		which is a contradiction. This completes the proof.
	\end{proof}

\end{document}